\documentclass[11pt,reqno,a4paper]{amsart}

\usepackage[T1]{fontenc}
\usepackage[utf8]{inputenc}
\usepackage[english]{babel}

\usepackage{amsmath}
\usepackage{amssymb}
\usepackage{amsfonts}

\usepackage{algorithm}
\usepackage{algpseudocode}

\usepackage{array}
\usepackage{multirow}
\usepackage{graphicx}
\usepackage{longtable}
\usepackage{lscape}
\usepackage{hhline}
\usepackage{adjustbox}
\usepackage{caption}
\usepackage{colortbl}
\usepackage[table]{xcolor}

\usepackage{tikz}
\usetikzlibrary{matrix,shapes,arrows,positioning,chains}

\usepackage{url}
\usepackage{comment}
\usepackage{textcomp}
\usepackage{chngcntr}
\usepackage{csquotes}
\usepackage{blindtext}
\usepackage{lmodern}

\theoremstyle{plain}
\newtheorem{thm}{Theorem}
\newtheorem{lem}{Lemma}
\newtheorem{cor}{Corollary}
\newtheorem{prop}{Proposition}

\theoremstyle{definition}
\newtheorem{ass}{Assumption}

\theoremstyle{remark}
\newtheorem{rem}{Remark}

\def\bepf{\begin{proof}}
\def\epf{\end{proof}}

\def\Box{{\hbox{\raisebox{0.0em}{\rlap{$\sqcap$}}\kern0em%
            \raisebox{-0.0em}{$\sqcup$}}} }

\makeatletter
\newcounter{subsubsubsection}[subsubsection]
\def\subsubsubsectionmark#1{}

\def\subsubsubsection{\@startsection
  {subsubsubsection}{4}{\z@}
  {-3.25ex plus -1ex minus -.2ex}
  {1.5ex plus .2ex}
  {\normalsize\bfseries}}
\def\l@subsubsubsection{\@dottedtocline{4}{4.8em}{4.2em}}
\makeatother

\newcolumntype{?}{!{\vrule width 1pt}}
\def\bfi#1{{\bf #1}}

\def\eeq{\end{equation}}
\def\lbeq#1{\begin{equation}\label{#1}}
\def\bary{\begin{array}}
\def\eary{\end{array}}

\def\D{\displaystyle}

\def\fct#1{\mathop{\rm #1}}

\def\argmin{\fct{argmin}}

\def\proj{\fct{proj}}

\newcommand{\BreuSS}{Breu\ss}

\begin{document}


\title[Reservoir ZCW Projected Subspace Search]%
{Reservoir Zero-Coordinatewise Projected Subspace Search for Minimization Over Sparse Symmetric Sets in Machine Learning}

\author{Morteza Kimiaei}
\address{Fakult\"at f\"ur Mathematik, Universit\"at Wien,
Oskar-Morgenstern-Platz 1, A-1090 Wien, Austria}
\email{morteza.kimiaei@univie.ac.at}

\author{Shima Shabani}
\address{Institute for Mathematics, Brandenburg University of Technology,
Platz der Deutschen Einheit 1, 03046 Cottbus, Germany}
\email{shima.shabani@b-tu.de}

\author{Michael \BreuSS}
\address{Institute for Mathematics, Brandenburg University of Technology,
Platz der Deutschen Einheit 1, 03046 Cottbus, Germany}
\email{breuss@b-tu.de}

\thanks{M. Kimiaei acknowledges financial support of the Austrian Science Foundation under
\url{https://doi.org/10.55776/PAT2747625}.}
\thanks{S. Shabani and M. Breu\ss\ acknowledge funding from the Federal Ministry of
Research, Technology and Space (BMFTR), Germany, through the project
``Digital GreenTech (DGT) -- Environmental Technology Meets Robotics''
(grant number 68542).}

\subjclass[2020]{Primary 90C30; Secondary 90C06, 65K05, 90C27, 90C59}

\keywords{Cardinality-constrained optimization, randomized optimization,
subspace technique, global convergence, convergence rate}

\begin{abstract}
\begin{sloppypar}
We study a class of nonconvex cardinality-constrained optimization problems
arising in sparse learning. These problems are NP-hard due to the combinatorial
nature of sparsity constraints. We introduce a Reservoir Zero-Coordinatewise
Projected Subspace Search ({\tt RZCW-PSS}) algorithm, a simplex-style method on
sparse manifolds that integrates coordinatewise search, symmetry-aware swap-based
support updates, randomized low-dimensional subspace exploration, and
zero-coordinatewise reservoir injection. The proposed method augments classical
coordinate and swap moves with sparse-compatible subspace searches constructed
from a dynamically maintained reservoir of previously accepted feasible points. A
key feature of the approach is a refined reservoir initialization strategy that
embeds sparse projection directly into a uniform sampling procedure, preserving
geometric diversity within the feasible set. The algorithm also includes an
optional support-identification safeguard that enforces full-support
stabilization under a fixed support-change decrease threshold. We establish that,
under the stated regularity, sampling, and subproblem-accuracy assumptions, every
full-support accumulation point of the {\tt RZCW-PSS} iterates is Beck--Hallak
zero-coordinatewise stationary almost surely; with the safeguard and full-support
initialization, this conclusion applies to all accumulation points. We further
prove a conditional local linear convergence rate after support stabilization and
derive the corresponding logarithmic local iteration complexity. Numerical
experiments on synthetic sparse learning problems demonstrate that
{\tt RZCW-PSS} improves robustness and solution quality while remaining
computationally competitive with Partial Simplex Search, Basic Feasible Search,
and Zero-Coordinatewise Search methods.
\end{sloppypar}
\end{abstract}

\maketitle

\begin{sloppypar}
\section{Introduction}

Sparsity-constrained optimization problems arise ubiquitously in modern data analysis, signal processing, machine learning, and statistics, where one seeks high-quality solutions that depend on only a small subset of decision variables. A widely studied modeling paradigm for enforcing sparsity is the incorporation of an explicit cardinality constraint, which limits the number of nonzero components of the decision vector. While such formulations offer strong interpretability and modeling fidelity, they also introduce fundamental algorithmic challenges due to the resulting nonconvexity and combinatorial structure of the feasible set. This paper focuses on the design and analysis of an efficient algorithm for solving smooth optimization problems under explicit cardinality constraints, with a particular emphasis on stationarity concepts and low-dimensional search mechanisms that strike a balance between computational tractability and solution quality.

\subsection{Problem Formulation}

We consider the \textbf{cardinality-constrained optimization problem} (CCOP)
\begin{equation}\label{eq:prodef}
	\min_{x \in C \cap C_s} f(x),
	\qquad
	C_s := \{x \in \mathbb{R}^n : \|x\|_0 \le s\},
\end{equation}
where $f:\mathbb{R}^n \to \mathbb{R}$ is continuously differentiable. Let $g(x) := \nabla f(x)$ be the gradient of $f$ at 
$x$. In this work, we propose the Reservoir Zero-Coordinatewise Projected Subspace Search
({\tt RZCW-PSS}) method for the CCOP \eqref{eq:prodef}..

\bfi{Assumptions on CCOP.}
The feasible set $C \subseteq \mathbb{R}^n$ is assumed to be closed, convex, and
symmetric in the sense of \cite{Beck2016}, as detailed in
Subsection~\ref{sec:SymmSet}. Typical examples of such symmetric sets include
$\mathbb{R}^n$, the nonnegative orthant, the unit simplex, $\ell_p$-norm
balls, and centered box-constrained sets; see \cite[Section~2]{suppMat} for details. These feasible regions are invariant
under permutations of the coordinate indices and, in certain cases, also under
componentwise sign changes. 

This symmetry plays a fundamental role in the characterization and efficient
computation of sparse projections \cite{Beck2016}, as discussed further in
Section~3 of \cite{suppMat}. In particular, it ensures that sparse projection
operators and coordinate permutations preserve feasibility, a property that is
exploited by the coordinatewise and support-exchange steps of the proposed
algorithm. A related convexification perspective is developed in
\cite{KimTawarmalaniRichard2022}, where permutation- and sign-invariant sets are
studied through majorization-based convex-hull descriptions, with applications
to sparse PCA and cardinality-constrained norm balls. In contrast, our use of
symmetry is algorithmic rather than relaxation-based: we exploit the symmetric
structure of \(C\) to characterize sparse projections, support exchanges, and
ZCW super-support subproblems within a feasible-search algorithm.

The sparsity set \(C_s\) enforces at most \(s\) nonzero components, rendering the
feasible region nonconvex and combinatorial. Consequently, problem
\eqref{eq:prodef} is NP-hard in general. We assume throughout that
\(g=\nabla f\) is Lipschitz continuous on bounded subsets of \(C\).

\bfi{Support sets.}
For any \(x \in \mathbb{R}^n\), we define the active and inactive index sets,
respectively, by
\begin{equation}\label{eq:supdef}
	I_1(x)=\operatorname{supp}(x) := \{i \in [n] : x_i \neq 0\},
	\qquad
	I_0(x) := \{i \in [n] : x_i = 0\},  
\end{equation}
where
\(\operatorname{supp}(\cdot)\) denotes the support operator, and
\([n]:=\{1,\dots,n\}\).

For a feasible point \(x\in C\cap C_s\), we say that \(x\) has saturated
support, or full sparse support, if \(\|x\|_0=s\), and unsaturated support if
\(\|x\|_0<s\). A super support of \(x\) is any index set
\(\mathcal L\subseteq[n]\) with \(|\mathcal L|=s\) such that $I_1(x)\subseteq\mathcal L$. Super supports play a central role in characterizing admissible support
modifications under the cardinality constraint.

\subsection{Related Works}

We organize the related literature into two main categories. We first review stationarity-based methods for cardinality-constrained optimization, and then discuss subspace methods that exploit low-dimensional search spaces.

\subsubsection{Stationarity-Based Methods for CCOP}

In the sparse optimization setting, classical first-order stationarity with respect
to all feasible directions is typically too strong to be enforced efficiently.
As a result, different algorithms aim to satisfy \textbf{restricted} first-order
optimality conditions of varying strength, ranging from projection-based fixed-point
conditions to coordinatewise and support-aware notions of stationarity. We briefly
review several methods proposed in \cite{Beck2013,Beck2016}, emphasizing their
algorithmic ideas and the specific stationarity concepts they enforce.

\bfi{Basic Feasible Search ({\tt BFS}).}
{\tt BFS} is based on the concept of \textbf{basic feasibility}, which can be viewed
as a restricted form of first-order stationarity adapted to cardinality constraints.
A point is said to be basic feasible if the gradient vanishes on its active support
(and on all coordinates when the support is not full). In contrast to classical
stationarity, basic feasibility does not require optimality with respect to all directions or support modifications; it only enforces first-order
conditions on a fixed support. The {\tt BFS} algorithm repeatedly constructs candidate
super supports of cardinality $s$ and solves smooth subproblems restricted to these
supports. Under mild assumptions, finite termination can be guaranteed. {\tt BFS}
serves as a fundamental building block for more advanced sparse optimization methods.

\bfi{Iterative Hard Thresholding ({\tt IHT}).}
{\tt IHT} is a first-order projection method for solving \eqref{eq:prodef}. Each
iteration consists of a gradient descent step followed by a hard thresholding
operator that retains the $s$ largest components in magnitude. From a stationarity
perspective, {\tt IHT} seeks fixed points of the projected gradient mapping onto the
nonconvex sparse set $C_s$. Under Lipschitz continuity of the gradient $g$, accumulation
points of {\tt IHT} satisfy an \textbf{$L$-stationarity} condition (see \cite[Section~4]{suppMat}), expressed as a
projection-based fixed-point relation. This condition represents a weak form of
first-order optimality, as it does not exclude descent along individual coordinates
or support-changing directions.

\bfi{Partial Sparse Simplex ({\tt PSS}).}
{\tt PSS} is a coordinate-descent-type algorithm designed to improve upon {\tt IHT}
by enforcing a stronger notion of stationarity. At each iteration, it performs exact
one-dimensional minimizations along coordinate directions while preserving sparsity.
By allowing controlled single-index support exchanges, {\tt PSS} explores a richer
set of feasible descent directions than {\tt IHT}. As a result, {\tt PSS} targets
\textbf{partial coordinatewise stationarity}, which requires optimality with respect to all
one-dimensional perturbations along active coordinates as well as all single-index
support swaps. This notion strictly strengthens $L$-stationarity by ruling out
coordinatewise and swap-based descent directions.

\bfi{Zero-Coordinatewise Search ({\tt ZCWS}).}
{\tt ZCWS}, also referred to as zero-coordinatewise search, enforces a stronger
stationarity concept that builds upon basic feasibility. Starting from a basic
feasible point, it systematically replaces poorly performing active coordinates with
more promising inactive ones based on gradient information. Each support update is
followed by a restricted minimization and a {\tt BFS} step to restore feasibility.
The resulting iterates satisfy \textbf{zero-coordinatewise (ZCW) stationarity} (see below), which
excludes descent along any coordinate direction, including after full
reoptimization on modified supports. This is the strongest stationarity notion among
the methods reviewed.

Overall, the methods {\tt BFS}, {\tt IHT}, {\tt PSS}, and {\tt ZCWS} differ primarily
in the strength of the stationarity conditions they enforce. {\tt BFS} targets
basic feasibility, a weak but foundational form of first-order stationarity on
fixed supports. {\tt IHT} relies on \(L\)-stationarity defined via a
projected-gradient fixed-point condition. {\tt PSS} enforces partial
coordinatewise stationarity with respect to both active-coordinate descent and
single-index support swaps. {\tt ZCWS} targets ZCW stationary points, which
correspond to the strongest notion in this hierarchy. This ordering explains both the increasing algorithmic complexity and the
progressively stronger optimality guarantees of these methods. The proposed
{\tt RZCW-PSS} method fits naturally within this hierarchy: it retains the
ZCW-stationarity target while enriching the search process through randomized
low-dimensional projected subspace exploration. For completeness and clarity, Section~\ref{sec:StatConcept} gives the precise
definition of the ZCW stationarity concept used in this paper. The broader
hierarchy of related stationarity notions, including basic feasibility,
\(L\)-stationarity, and ZCW stationarity, is summarized in
Figure~1 of \cite{suppMat}, where arrows indicate implications from stronger to
weaker conditions.

Beyond the stationarity-based methods reviewed above, several other algorithmic
families have been developed for cardinality-constrained and sparse optimization.
Greedy methods, such as matching-pursuit and forward/backward selection
schemes, are computationally attractive, but may be sensitive to noise,
correlations, and early variable-selection errors
\cite{cvetkovic2022greedy,wen2025randomized}. Convex relaxation approaches,
including basis pursuit, {\tt LASSO}, and elastic-net formulations, replace the
\(\ell_0\)-constraint by tractable convex surrogates, but typically produce
biased or only approximately sparse solutions
\cite{Chen1998,Esmaeili2018,tibshirani1996regression}. Thresholding and
projected-gradient methods, including {\tt IHT} and its nonmonotone or
support-changing variants, provide scalable first-order alternatives but are
often sensitive to stepsize rules and problem conditioning
\cite{bergamaschi2025probabilistic,Blumensath2009,hu2025convergence,lu2013sparse,zhao2020optimal}.
Penalty-decomposition and augmented-Lagrangian frameworks form another important
class, since they separate smooth objective components from sparsity or
geometric constraints \cite{KanzowLapucci,Lapucci2020,lu2013sparse}. These
methods can also be combined with line-search strategies and quasi-Newton
updates to improve practical efficiency and scalability
\cite{PDQN}. Finally, mixed-integer and combinatorial formulations can certify global
optimality for some sparse models, but their scalability is limited by the
underlying combinatorial search space
\cite{Bertsimas2016,Burdakov2016}. Recent specialized branch-and-bound frameworks for sparse regression exploit first-order optimization, warm starts, active sets, gradient screening, and problem-specific dual bounds to improve scalability substantially \cite{HazimehMazumderSaab2022}. These methods provide global certificates for \(\ell_0\ell_2\)-regularized sparse regression, whereas the focus of the present work is different: we target general smooth cardinality-constrained problems and develop a stationarity-oriented randomized subspace method rather than an exact global optimization procedure.

\subsubsection{Subspace Methods}

Subspace methods play a central role in practical optimization and can be broadly classified into deterministic approaches \cite{kimiaei2022new,LMBOPT,kimiaei2023new} and randomized approaches \cite{Cartis2022LS,cartis2024randomized,cartis2025random}. These methods are employed either to construct local models within deterministic or randomized subspaces, whose solutions provide effective search directions, or to generate subspaces spanned by collections of vectors such as descent directions, differences of iterates, gradient differences, or related quantities. As emphasized in the comprehensive review of subspace techniques by Yuan \cite{Yuan2014}, the effectiveness of subspace methods in large-scale optimization stems from their ability to replace full-dimensional subproblems with lower-dimensional ones, thereby significantly reducing computational cost while retaining favorable convergence properties. The review systematically covers subspace methods for unconstrained and constrained optimization, nonlinear equations, nonlinear least squares, and matrix optimization, and highlights that the construction and selection of appropriate subspaces are the essential components governing algorithmic performance. In the context of cardinality-constrained optimization, subspace searches enable simultaneous multi-coordinate exploration without explicit enumeration of supports and can significantly enhance the ability of coordinate-based methods to escape poor stationary points.

Projection-free sparse convex optimization provides a related but distinct low-dimensional search paradigm. Frank--Wolfe, or conditional-gradient, methods move toward atoms obtained from linear minimization oracles and maintain iterates as convex combinations of few atoms, thereby naturally producing sparse or low-rank structures \cite{Jaggi2013}. This philosophy is related to our use of low-dimensional exploratory steps, but the setting and mechanism are different: Frank--Wolfe operates over convex hulls of atoms and avoids projections, whereas {\tt RZCW-PSS} works directly with the nonconvex feasible set \(C\cap C_s\) and uses sparse projections, support exchanges, and ZCW-aware reservoir injections.

\subsection{Our Contribution}

Building on the stationarity hierarchy underlying {\tt IHT}, {\tt PSS}, {\tt BFS}, and {\tt ZCWS}, the proposed {\tt RZCW-PSS} method augments classical simplex-style coordinate updates with low-dimensional projected subspace searches constructed from a reservoir of previously accepted feasible points and ZCW-aware support injections.

A key algorithmic component of {\tt RZCW-PSS} is a novel reservoir initialization and update mechanism based on a refined uniform sampling procedure.
The refined uniform sequence, denoted by {\tt usequence} and introduced in \cite{MATRS}, generates well-separated points in a simple ambient domain, typically the hypercube $[-1,1]^n$, without
regard to problem-specific constraints.
In {\tt RZCW-PSS}, we generate a projected uniform sequence (see Section~8 of \cite{suppMat}), called {\tt p-usequence}, by preserving the max-min separation property of {\tt usequence}
and coupling it with sparse feasibility through the application of the sparse projection operators of \cite[Algorithms 1--4]{Beck2016} to each generated sample, yielding a reservoir of points in $C \cap C_s$ that remains well distributed after projection and is directly usable by the algorithm.

The proposed framework consists of the following components \textbf{(i)}--\textbf{(vii)}:

\bfi{(i) Initialization via Structured Exploration and {\tt BFS} Warm Start.}
The algorithm initializes the reservoir using the {\tt p-usequence} procedure,
producing a diverse collection of feasible points in $C \cap C_s$ without relying
on local optimality. This structured exploration mitigates sensitivity to the
initial point and provides broad coverage of the sparse feasible region. A
{\tt BFS} warm-start phase is then applied to refine feasibility and objective
value before the main iterations. The integration of {\tt p-usequence}
initialization with a {\tt BFS}-based warm start is novel.

\bfi{(ii) Coordinatewise Search.}
The method employs coordinatewise descent steps based on one-dimensional
optimization along individual coordinates. These steps provide inexpensive and
reliable local improvements while preserving sparse feasibility. They enforce
partial coordinatewise stationarity on both active and inactive coordinates when
the support is not saturated. This component follows established partial simplex
search and coordinate descent principles.

\bfi{(iii) Support Swapping.}
When the support is full, the algorithm performs symmetry-aware single-index
support swaps. Weak active coordinates are replaced by inactive coordinates
exhibiting the strongest descent potential. This mechanism allows the method to
escape suboptimal supports that block coordinatewise descent. The swap strategy
itself is standard, but its role within the unified {\tt RZCW-PSS} framework is
refined.

\bfi{(iv) Randomized Exploration.}
Beyond coordinatewise and swap moves, the algorithm performs randomized
exploration over low-dimensional sparse subspaces. These subspaces are constructed from a reservoir of previously accepted feasible
points; each basis direction is restricted to the current active support together
with at most one inactive coordinate, while feasibility of general subspace trial
points is restored by projection onto \(C\cap C_s\). Subspace optimization enables simultaneous multi-coordinate
adjustments that cannot be captured by one-dimensional searches. This randomized
sparse subspace mechanism is new.

\bfi{(v) Global Refinements.}
When local mechanisms fail to produce improvement after approximate stationarity,
the algorithm activates a global refinement step based on {\tt p-usequence}.
A fresh collection of well-separated feasible points in $C \cap C_s$ is generated
to reinitialize the reservoir and restart the search from a new basin of
attraction. This refinement strategy preserves sparse feasibility and avoids blind
random restarts. Such a controlled, projection-aware global refinement mechanism
is novel.

\bfi{(vi) ZCW-Aware Reservoir Injection.}
The algorithm includes a probabilistic reservoir enrichment mechanism designed to
detect violations of  ZCW stationarity in the saturated regime.
Refined probe points are generated by approximately solving restricted
subproblems over ZCW super supports and are inserted into the reservoir
subject to diversity pruning. Under the uniform support-hitting condition used in
the convergence analysis, the mechanism ensures that ZCW super supports
associated with stationarity violations are sampled infinitely often along
saturated subsequences with probability one. The resulting ZCW-aware reservoir
injection is a new contribution linking randomized exploration to
support-based stationarity detection.

\bfi{(vii) Inexact Acceptance Rule and Support-Identification Safeguard.}
The algorithm uses an inexact candidate-set acceptance rule that selects the next
iterate from the finite set of generated feasible candidates up to a prescribed
outer tolerance. This rule separates the outer candidate-selection error from the
inner restricted-solver error used in the ZCW reservoir injection step. In
addition, we introduce an optional support-identification safeguard in
the last step of the {\tt RZCW-PSS} algorithm. When activated, this safeguard filters the candidate set before
acceptance: every admissible candidate must have full support and active
components uniformly bounded away from zero, and any candidate that changes the
current support must produce a fixed objective decrease. This safeguard is not
part of the ZCW stationarity definition; it is an algorithmic device used
only in the local convergence analysis. It guarantees eventual full-support
stabilization under full-support initialization, which allows the almost-sure
ZCW stationarity result to extend from full-support accumulation points to all
accumulation points in the safeguarded implementation.

\bfi{Convergence.} From a theoretical perspective, we establish basic convergence
properties of the proposed method and prove that the ZCW-aware reservoir injection mechanism is sufficient to detect  ZCW violations in the saturated regime. In particular, under mild regularity assumptions and a uniform support-hitting condition, every full-support
accumulation point of the {\tt RZCW-PSS} iterates is  ZCW stationary almost surely. Moreover, when the support-identification safeguard is imposed together with full-support initialization, the generated sequence eventually remains on one fixed full support. Consequently, in the safeguarded implementation, every accumulation point is full-support and hence ZCW stationary almost surely.

\bfi{Convergence Rates and Complexity.}
Global iteration complexity bounds are generally unavailable for
cardinality-constrained optimization due to the nonconvex and combinatorial structure of the feasible set. Instead, our analysis separates the stochastic support-exploration phase from the deterministic local refinement phase. 

For the exploration phase, we provide an expected hitting-time bound for a fixed violated ZCW super support under a uniform support-hitting probability condition. This result should not be interpreted as a global finite-time optimality. Under the support-identification safeguard and full-support initialization, the support stabilizes after finitely many iterations. Once the stabilized support enters the local neighborhood where restricted strong
convexity, the restricted projected-gradient error bound, and the local
candidate-richness condition hold, the refinement phase converges linearly,
yielding the logarithmic local iteration complexity derived in
Section~7 of \cite{suppMat}. This transition from stochastic support exploration
to deterministic restricted refinement is a key feature of the proposed hybrid
reservoir framework.

\bfi{Numerical Results.}
From a numerical standpoint, the proposed {\tt RZCW-PSS} framework is designed
to recover sparse solutions that are structurally meaningful relative to those
obtained by existing methods, while achieving superior optimization quality; see,
for example, Figures~3--8 in \cite{suppMat}. The reservoir mechanism enables the construction of informative low-dimensional search subspaces from sparse, reservoir-based directions, allowing randomized subspace searches to exploit nonlocal information while preserving compatibility with the cardinality constraint. Rather than targeting exact support recovery, the numerical goal is to attain high-quality sparse points with low objective values while maintaining substantial overlap with reference supports. Empirical results on a collection of
synthetic sparse learning problems confirm that {\tt RZCW-PSS} improves
robustness and solution quality relative to classical methods such as {\tt ZCWS},
without a substantial increase in per-iteration computational complexity.


\section{Preliminaries}

This section recalls the main notation, symmetry classes, and ZCW stationarity concept for cardinality-constrained optimization that are used throughout the paper. These notions are based primarily on the
framework of Beck and Hallak~\cite{Beck2016}. We include them here to make the presentation self-contained and to clarify how the proposed {\tt RZCW-PSS} method uses sparse projections, super supports, and  stationarity conditions.

\subsection{Symmetric Sets and Symmetry-Aware {\tt BFS} Ranking}\label{sec:SymmSet}

Let $\mathfrak S_n$ denote the group of all permutations of the index set $[n] = \{1,\dots,n\}$.
Given a vector $x\in\mathbb R^n$ and a permutation $\pi\in\mathfrak S_n$, we define the permuted vector
$x^\pi\in\mathbb R^n$ by rearranging the components of $x$ according to $\pi$, namely $(x^\pi)_i := x_{\pi(i)}$ for all $i\in [n]$.

Among all permutations, we distinguish those that reorder the components of a vector in
nonincreasing order. For \(x\in\mathbb R^n\), let \(\widetilde{\mathfrak S}(x)\) denote the set of
permutations \(\pi\in\mathfrak S_n\) satisfying $x_{\pi(1)}\ge x_{\pi(2)}\ge\cdots\ge x_{\pi(n)}$. Any \(\pi\in\widetilde{\mathfrak S}(x)\) is called a sorting permutation of \(x\).

We now describe several symmetry properties of convex sets.
Let $C\subseteq\mathbb R^n$ be closed and convex.
The set $C$ is called type-1 symmetric if it is invariant under arbitrary permutations of coordinates,
that is, $x^\pi\in C$ for all $x\in C$ and all $\pi\in\mathfrak S_n$.
It is called nonnegative if every vector in $C$ has nonnegative components.
Finally, $C$ is said to be type-2 symmetric if it is type-1 symmetric and, in addition, invariant under
componentwise sign changes, meaning that $x\circ y\in C$ for all $x\in C$ and all sign vectors
$y\in\{-1,1\}^n$, where $(x\circ y)_i:=x_i y_i$.

Throughout this work, we consider closed and convex sets $C$ that are either nonnegative type-1
symmetric or type-2 symmetric.
To unify notation in both cases, we use the mapping $\textbf{p}:\mathbb R^n\to\mathbb R^n$ defined by
\begin{equation}\label{e.pdef}
	\textbf{p}(x)=
	\begin{cases}
		x, & \text{$C$ nonneg. type-1},\\
		|x|, & \text{$C$ type-2},
	\end{cases}
\end{equation}
where $|x|$ denotes the vector of componentwise absolute values.
Sorting permutations are then taken with respect to $\textbf{p}(x)$, that is,
$\pi\in\tilde{\mathfrak S}(\textbf{p}(x))$.

{\bf Symmetry-aware {\tt BFS} ranking.} Let $x\in \Omega:= C \cap C_s$ and $g(x)=\nabla f(x)$. Inactive indices are ranked according to the score
\begin{equation}\label{e.sigmap}
	\sigma_i(x) := \mathbf p_i(-g(x)),
	\qquad i\in I_0(x),
\end{equation}
which measures the magnitude of first-order optimality violation in a manner
consistent with the symmetry of $C$. It is easily obtained from \eqref{e.pdef} that, for nonnegative type-1 symmetric sets, \(\sigma_i(x)=-g_i(x)\), while for
type-2 symmetric sets, \(\sigma_i(x)=|g_i(x)|\). The \texttt{BFS} ranking orders inactive indices in decreasing order of
\(\sigma_i(x)\), thereby selecting first the indices with the largest
symmetry-aware descent score.

Sampling from the ranked inactive set assigns strictly positive probability to
each of the top-ranked indices. The scores are continuous in \(x\) whenever \(f\) is continuously differentiable.
They are symmetry-aware through the use of \(\mathbf p(\cdot)\), but they are not
invariant under coordinate permutations or sign changes unless additional symmetry
assumptions are imposed on \(f\). The induced ranking may be set-valued in the presence of ties; ties are
resolved by the prescribed randomized selection rule. 

\subsection{ZCW stationary point}\label{sec:StatConcept}

In the introduction, coordinatewise stationarity is mentioned only informally to
motivate the limitations of classical coordinatewise sparse methods. Here, we focus exclusively on  zero-coordinatewise stationarity,
which is the optimality concept targeted by {\tt RZCW-PSS} and used in the
convergence theory. Precise definitions of other stationarity notions and their
relationships are deferred to the supplemental material \cite{suppMat}.

Suppose that
\(C\subseteq\mathbb R^n\) is either nonnegative type-1 symmetric
or type-2 symmetric. Let \(x^*\in\Omega\) be a basic feasible point, recall \(\mathbf p(\cdot)\) from~\eqref{e.pdef} and define the set of least significant active indice by
\begin{equation}\label{eq:Ddef}
	Q(x^*):=
	\argmin_{q\in I_1(x^*)} \mathbf p_q(x^*),
\end{equation}
and define the indices
\begin{equation}\label{eq:ijdef}
	q^*:=q(x^*)\in
	\argmin_{q\in Q(x^*)}
	\mathbf p_q(-g(x^*))
	~~\text{and}~~
	j^*:=j(x^*)\in
	\argmin_{j\in I_0(x^*)}
	\left\{-\mathbf p_j(-g(x^*))\right\}.
\end{equation}

Moreover, let $\pi\in
\widetilde{\mathfrak S}
\!\left(
-\mathbf p(-g(x^*))
\right)$ and choose \(k\in[n]\) such that $
\left|
T^*
\right|
=
s
$, where
\begin{equation}\label{eq:supsubzcw}
	T^*
	:=
	\left(
	S^\pi_{[k,n]}
	\cup I_1(x^*)
	\cup \{j^*\}
	\right)
	\setminus \{q^*\}.   
\end{equation}
Then \(x^*\) is called a ZCW  point if $f(x^*)
\le
\min
\left\{
f(y):
y\in C,\;
I_1(y)\subseteq T^*
\right\}$.

ZCW stationarity is a support-based condition. After the active and inactive
indices are selected, a super support set \(T^*\) is constructed, and \(x^*\) is
compared with all feasible points whose support is contained in \(T^*\). In contrast, the simple-CW criterion is based on a single coordinatewise
exchange. Given an active index \(q\) and an inactive index \(j\) selected according to
\eqref{eq:ijdef}, the corresponding simple-CW candidate is obtained by
replacing the active coordinate \(q\) with the inactive coordinate \(j\) in a
manner consistent with the symmetry structure of the set \(C\). The resulting
candidate preserves both the sparsity level and the magnitude of the
exchanged entry. Thus,
ZCW (full-CW) is stronger than simple-CW because it tests optimality over an entire
restricted sparse subproblem, rather than only over the simple swap candidates.


\section{Our {\tt RZCW-PSS} Method}

In this section, we describe the general framework of {\tt RZCW-PSS}, for solving \eqref{eq:prodef}.

\subsection{Subspace Optimization over Sparse Manifolds}\label{sec:res-cw}

\bfi{Reservoir-induced projected subspaces and reduced optimization.}
The proposed algorithm incorporates searches over low-dimensional projected
subspaces constructed from a finite reservoir of feasible points. The reservoir
\(\mathcal R\subset C\cap C_s\) is initialized using the refined uniform
sampling procedure based on {\tt p-usequence}; see Section~8 of \cite{suppMat}.
This procedure generates well-separated feasible points according to a max-min
distance criterion and projects them onto \(C\cap C_s\) using the sparse
projection operators for symmetric convex sets described in Algorithms~1--4
of \cite{Beck2016}, see also Section~3 of \cite{suppMat}. The resulting reservoir provides a geometrically diverse
set of feasible reference points and allows the subspace mechanism to generate
nonlocal exploration directions that do not rely solely on local information at
the current iterate.

Given a current iterate \(x\in C\cap C_s\), let
\(\mathfrak U(x;\mathcal R)\) denote the collection of subspaces generated from
\(\mathcal R\) with respect to $x$. For each selected reservoir point \(r^\ell\in\mathcal R\), a
preliminary reservoir direction \(d_\ell\) is obtained from the displacement
\(r^\ell-x\) by retaining the components indexed by
\(I_1(x)\cup\{j_\ell\}\), for some \(j_\ell\in I_0(x)\). Thus, using
\eqref{eq:supdef},
\[
\operatorname{supp}(d_\ell)
\subseteq
I_1(x)\cup\{j_\ell\},
\qquad
\forall\ell\in[m].
\]
Hence, each preliminary direction is associated with the current active support
together with at most one inactive coordinate, and is compatible with a single
potential support exchange.

The sparsity restriction is imposed on the preliminary directions \(d_\ell\),
before orthonormalization. Nonzero preliminary directions are normalized and
assembled into a matrix, and an orthonormal basis is then extracted using a
numerically stable orthonormalization procedure, such as modified Gram-Schmidt,
QR factorization, or singular value decomposition. We denote the resulting basis
by
\begin{equation}\label{eq:Udef}
	U=[u_1,\ldots,u_m]\in\mathbb R^{n\times m},
\end{equation}
where \(m\) is the subspace dimension. The corresponding reservoir-induced
subspace is
\[
\mathcal U=\operatorname{span}(U)
=
\operatorname{span}\{u_1,\ldots,u_m\}.
\]

The orthonormalized basis vectors \(u_\ell\) span the same subspace generated by
the sparse preliminary directions \(d_\ell\), but the notation \(u_\ell\) is
reserved for the normalized/orthonormalized basis used in the subspace search. A
general vector in the subspace can be written as
\[
u=\sum_{\ell=1}^m a_\ell u_\ell\in\mathcal U,
\qquad
a_\ell\in\mathbb R,\quad \forall\ell\in[m],
\]
or equivalently \(u=Ua\), where $a=(a_1,\ldots,a_m)^\top\in\mathbb R^m$. Such a vector may combine basis directions associated with different inactive
coordinates. Hence, the subspace perturbation \(x+u\) need not itself be sparse
feasible before projection.

Given the orthonormal basis \(U\), the method performs a reduced smooth
optimization over the subspace rather than sequential one-dimensional line
searches. Specifically, it considers the reduced model $\phi(\alpha):=f(x+U\alpha)$ with
$\alpha\in\mathbb R^m$ with exact reduced gradient $\nabla\phi(\alpha)=U^\top g(x+U\alpha)$. The inner optimization is terminated at an approximate stationary point
\(\alpha^\ast\) of the reduced model, or at a point that provides sufficient
decrease in the projected trial objective. The resulting trial point is then
projected onto the feasible set to restore both sparse and convex feasibility:
\begin{equation}\label{eq:projected-subspace-candidate}
	x^+ := \proj_{C\cap C_s}(x+U\alpha^\ast).
\end{equation}
Thus, the subspace optimization step is a projected search mechanism: the
reduced optimization is carried out in the ambient subspace, while feasibility is
enforced at the trial-point level through projection onto
\(\Omega=C\cap C_s\).

This subspace step enables simultaneous multi-coordinate adjustments and is used
solely for continuous exploration. It is not identified with the coordinatewise,
simple-CW, or ZCW optimality conditions. In particular, swap operations are
applied only through the coordinatewise/simple-CW mechanism described below, and
are not interpreted as random subspace moves.

To prevent degeneration of the subspace exploration mechanism, the reservoir is
updated dynamically using objective-based selection and a minimum-distance
diversity criterion. This keeps the stored points both competitive in objective
value and well separated in the decision space. The reservoir update rule,
including diversity pruning, is described in Section~15 of \cite{suppMat}; it is
the mechanism that keeps \(\mathcal R\) informative and underlies the uniform
probability bounds used in the complexity analysis of Section~7 of
\cite{suppMat}.

For active-coordinate refinement, for each \(i \in I_1(x^k)\), we define
\begin{equation}\label{eq:active-coordinate-value}
	f_i^k
	:=
	\min_{t \in \mathbb{R}}
	f(x^k + t e_i).
\end{equation}
Thus, \(f_i^k\) is the minimum objective value attainable by varying the
\(i\)th coordinate.
So, the index used for active-coordinate refinement and the refinement value are
\begin{equation}\label{eq:active-coordinate-index}
	i_1^k \in \argmin_{i \in I_1(x^k)} f_i^k  \quad \text{and} \quad D_1^k := f_{i_1^k}^k.
\end{equation}

When \(\|x^k\|_0=s\), from \eqref{eq:ijdef}, the  simple-CW indices are reconsidered as
\begin{equation}\label{eq:simple-cw-indices}
	q^k:=q(x^k),
	\qquad
	j^k:=j(x^k),
\end{equation}
and the simple-CW, \(+\)-swap, candidate is defined by
\begin{equation}\label{eq:simple-cw-plus-generic}
	x^{\,q^k\to j^k}_+
	:=
	x^k - x_{q^k} e_{q^k} + x_{q^k} e_{j^k}.
\end{equation}
For type-2 symmetric sets, the additional sign-symmetric candidate (called \(-\)-swap) is
\begin{equation}\label{eq:simple-cw-minus-generic}
	x^{\,q^k\to j^k}_-
	:=
	x^k - x_{q^k} e_{q^k} - x_{q^k} e_{j^k} .
\end{equation}
Then, the swap candidates
are defined at \(x^k\) as
\begin{equation}\label{eq:simple-cw-candidate-set-generic}
	\mathcal S^k_{\rm CW}
	:=
	\mathcal S_{\rm CW}(x^k):=
	\begin{cases}
		\left\{
		x^{\,q^k\to j^k}_+,
		x^{\,q^k\to j^k}_-
		\right\},
		& \text{$C$ type-2},\\[2mm]
		\left\{
		x^{\,q^k\to j^k}_+
		\right\},
		& \text{$C$ nonneg. type-1}.
	\end{cases}
\end{equation}

So, the simple-CW swap value is
\begin{equation}\label{eq:simple-cw-value}
	D_2^k
	:=
	\min\left\{
	f(y):y\in \mathcal S^k_{\rm CW}
	\right\}.
\end{equation}

Therefore, the candidate yielding the smaller objective value is accepted, provided
that it gives sufficient decrease. Specifically, if \(D_1^k < D_2^k\), then
\begin{equation}\label{eq:active-coordinate-update}
	x^{k+1} :=x^k + t_1^k e_{i_1^k},
\end{equation}
where
\begin{equation}\label{eq:active-coordinate-stepsize}
	t_1^k \in
	\argmin_{t \in \mathbb{R}}
	f(x^k + t e_{i_1^k}).
\end{equation}
Otherwise, \(x^{k+1}\) is set to the best simple-CW swap candidate:
\begin{equation}\label{eq:simple-cw-update}
	x^{k+1}
	\in
	\argmin
	\left\{
	f(y):\,
	y\in \mathcal S^k_{\rm CW}
	\right\}.
\end{equation}

{\bf ZCW-Aware Reservoir Injection Mechanism.}   To ensure that the reservoir remains sufficiently informative for generating
ZCW-related support-exchange candidates, the algorithm incorporates a
\textbf{randomized ZCW-aware reservoir enrichment mechanism}. Specifically, at
every iteration $k$ with $\|x^k\|_0=s$, the algorithm activates a ZCW-aware injection step with probability $p^k\in(0,1)$ satisfying
\begin{equation}\label{eq:infpk}
	\inf_{k\ge 1} p^k > 0,
\end{equation}
and, in particular,
\begin{equation}\label{eq:pklb}
	p^k\ge p_{\mathrm{inj}}>0,
\end{equation}
where $p_{\mathrm{inj}}\in(0,1)$ is a user-specified tuning parameter. The
condition $p^k\in(0,1)$ excludes degenerate cases in which the injection step is
either never activated (\(p^k=0\)) or deterministically enforced (\(p^k=1\)), while the uniform lower
bound \eqref{eq:infpk} guarantees infinitely many injection events almost surely.

Each activation triggers a symmetry-aware {\tt BFS} selection of active and
inactive indices using the  simple-CW index selection rule, followed
by the construction of a ZCW super support. A restricted subproblem is then
approximately solved by an appropriate inner solver. When the restricted model has
additional structure, such as a composite convex form, a {\tt FISTA}-type
method~\cite{beck2009fast} may be used; otherwise, the convergence analysis only
requires the stated restricted-solver accuracy and does not rely on the convexity of
the restricted objective. While condition~\eqref{eq:infpk} is primarily used to establish
almost-sure detection of ZCW violations, it also implies a uniform lower bound on
the probability of successful support-correcting injections, which underlies the
expected hitting complexity analysis in Section~7 of \cite{suppMat}. The ZCW-aware reservoir injection mechanism consists of five components:
support selection, the probabilistic sampling model, the uniform exploration
condition, restricted-subproblem solvability, and the restricted solve with
reservoir update, described in (i)--(v) below.

{\bf (i) BFS-Based Support Selection.}
Conditioned on the activation event, we recall \(Q(\cdot)\) from
\eqref{eq:Ddef}, written as \(Q(x^k)\) at \(x^k\). We select an active index
\(q^k\) and an inactive index \(j^k\) according to \eqref{eq:ijdef}. Let
\[
\pi^k
\in
\widetilde{\mathfrak S}
\left(
-\mathbf p(-g(x^k))
\right),
\]
and choose \(k_0\in[n]\) so that the resulting ZCW super support has cardinality
\(s\). Then, using the construction in \eqref{eq:supsubzcw}, define
\begin{equation}\label{eq:supconrul}
	T^k
	:=
	\left(
	S^{\pi^k}_{[k_0,n]}
	\cup I_1(x^k)
	\cup \{j^k\}
	\right)
	\setminus\{q^k\}.   
\end{equation}
Here, \(q^k\) is selected among the least significant active components, with
the gradient tie-breaker, while \(j^k\) is selected among inactive indices
according to the symmetry-aware score.

{\bf (ii) Probabilistic Model and Sampling Rule.}
We define a probability space \((\Xi,\mathcal F,\Pr)\), where \(\Xi\) is the
sample space of all random outcomes generated by the algorithm, \(\mathcal F\)
is a \(\sigma\)-algebra of events, and \(\Pr\) is the probability measure. The
randomness comes from the reservoir-injection decisions and from the randomized
support-selection rule.

Let \((\mathcal F^k)_{k\ge0}\) denote the natural filtration of the algorithm.
That is, each \(\mathcal F^k\) is a \(\sigma\)-algebra representing all
information available after completion of iteration \(k\) and
\[
\mathcal F^0\subseteq \mathcal F^1\subseteq \cdots \subseteq \mathcal F^k .
\]
A random object is said to be \(\mathcal F^k\)-measurable if its value is fully
determined by the information available up to iteration \(k\).

Equivalently, before the fresh random choices of iteration \(k+1\), the current
iterate \(x^{k+1}\) is \(\mathcal F^k\)-measurable. Thus, at iteration \(k\),
the conditioning \(\sigma\)-algebra \(\mathcal F^{k-1}\) contains the current
iterate \(x^k\), while the new random choices of iteration \(k\) have not yet
been made.

We define the event
\begin{equation}\label{e:ak}
	A^k := \{\text{ZCW-aware injection is activated at iteration } k\}.
\end{equation}
Conditional on \(\mathcal F^{k-1}\),
the activation event \(A^k\) is generated first. Conditional on
\((\mathcal F^{k-1},A^k)\), the ZCW super support \(T^k\) is then generated
according to the support-selection distribution specified below. In particular,
\(A^k\), \(T^k\), and the event
\[
A^k\cap\{T^k=T\}
\]
are \(\mathcal F^k\)-measurable. This convention justifies the conditional
probability factorization used in \eqref{eq:prob-factorization} in the proof of
Theorem~\ref{thm:RZCW-PSS-ZCW}, below.

{\bf (iii) Uniform Exploration Guarantee.}
The activation probability \(p^k\) in \eqref{eq:infpk} controls whether the
ZCW-aware reservoir injection mechanism is performed at iteration \(k\).
Conditional on this activation, the inactive coordinate \(j^k\in I_0(x^k)\) is
selected according to a probability distribution
\(\{p_j^k\}_{j\in I_0(x^k)}\), where
\[
p_j^k
:=
\Pr\big(j^k=j\mid \mathcal F^{k-1},A^k\big).
\]
We assume that this distribution satisfies the coordinate-level lower bound
\begin{equation}\label{eq:uniform-hitting}
	p_j^k
	\;\ge\;
	\frac{c_j}{|I_0(x^k)|},
	\qquad \forall j \in I_0(x^k),
\end{equation}
for some constant \(c_j\in(0,1]\) independent of \(k\). This condition ensures
that no inactive coordinate is systematically excluded once the ZCW-aware
reservoir injection mechanism is activated.

For the almost-sure ZCW convergence analysis, however, the required condition is
the corresponding support-level exploration property. Namely, for every ZCW
super support \(T\) that can be generated by the construction at \(x^k\), the
ZCW-aware reservoir injection mechanism assigns probability
\begin{equation}\label{eq:support-hitting}
	\Pr(T^k=T\mid\mathcal F^{k-1},A^k)\ge c_0,
\end{equation}
for some constant \(c_0>0\) in the local neighborhood under consideration. This
support-level lower bound can be enforced by adding a uniform exploration
component over the admissible choices used to construct \(T^k\), including the
inactive index, active index, sorting tie-breaks, and cutoff defining the ZCW
super support. This is the condition used in the almost-sure ZCW convergence
proof.

\bfi{(iv) Restricted ZCW subproblem solvability.}
Following the  ZCW framework, we assume that every
restricted ZCW subproblem generated by the algorithm is solvable. That is, for
every ZCW super support \(T\) generated by the support-construction rule \eqref{eq:supconrul}, the
problem
\[
\min
\left\{
f(z):
z\in C,\;
\operatorname{supp}(z)\subseteq T
\right\}
\]
attains its minimum.

This assumption is used only for the ZCW restricted solves. It is distinct from
the well-definedness of Euclidean projections, which follows from closedness, and uniqueness of convex projections. The assumption is automatic, for example, when the
restricted feasible set is compact, as in simplex, box, and bounded \(\ell_p\)-ball
constraints.

{\bf (v)  Restricted Subproblem and Reservoir Update.}
The ZCW-aware restricted solve is distinct from the randomized subspace
optimization step described above. The projected subspace candidate $x^+$ (defined by \eqref{eq:projected-subspace-candidate}) is generated from a reservoir-induced basis \(U\). In contrast, the ZCW-injected
candidate \(r^k\) is generated by solving a restricted sparse subproblem over
the ZCW super support \(T^k\), and it does not depend on the subspace basis
\(U\).

Given the ZCW super support \(T^k\) in \eqref{eq:supconrul}, a refined probe
point \(r^k\in C\cap C_s\) is obtained by approximately solving the restricted
problem over \(T^k\) using a suitable inner solver. Specifically, the computed
point satisfies
\begin{equation}\label{eq:defxik_final_Tk}
	f(r^k)
	\le
	\min
	\left\{
	f(z):
	z\in C,\;
	\mathrm{supp}(z)\subseteq T^k
	\right\}
	+\varepsilon_T^k,
\end{equation}
where \(\varepsilon_T^k\ge0\) denotes the restricted subproblem optimization
error. The sequence \(\{\varepsilon_T^k\}_{k\geq 0}\) is controlled through the stopping
criterion of the inner solver.

The restricted solver is initialized from a warm-start point \(\tilde x^k\).
This point is obtained by preserving the components of \(x^k\) on the generated
support \(T^k\), removing the selected departing active coordinate \(q^k\), and
transferring its value to the selected incoming inactive coordinate \(j^k\). In
particular,
\[
\tilde{x}^k_{j^k}=x^k_{q^k},
\qquad
\tilde{x}^k_{q^k}=0.
\]
Thus, the warm start mimics the simple support exchange \(q^k\to j^k\): the
active component at \(q^k\) is removed, while the inactive coordinate \(j^k\) is
activated using the transferred value.

The resulting point \(r^k\) is inserted into the current candidate set and into
the reservoir, subject to objective-based and distance-based pruning rules.
Inserting \(r^k\) into the current candidate set is essential: the convergence
analysis uses the fact that an improvement found by the restricted ZCW subproblem, once generated, is
immediately available for acceptance by the algorithm. The pruning rules ensure
that the reservoir remains a diverse global memory of the search space,
preventing it from becoming locally degenerate.

\bfi{Acceptance rule and support-identification safeguard.}
At iteration \(k\), all feasible candidates generated by the algorithm are
collected in a finite candidate set \(\Omega^k\subseteq C\cap C_s\), with
\(x^k\in\Omega^k\). Without the support-identification safeguard, the next
iterate is selected to satisfy
\begin{equation}\label{eq:accept-no-safeguard}
	f(x^{k+1})
	\le
	\min_{z\in\Omega^k} f(z)+\xi^k .
\end{equation}
This rule allows the algorithm to choose any candidate whose objective value is within \(\xi^k\ge0\) of the best generated candidate.

When the optional support-identification safeguard is imposed, the candidate set is
filtered before the acceptance step. Only candidates belonging to
\begin{equation}\label{eq:Omega-hat-safeguard}
	\widehat{\Omega}^k
	:=
	\left\{
	\begin{aligned}
		y\in\Omega^k:\;&
		\|y\|_0=s,\quad
		\min_{i\in I_1(y)}|y_i|\ge\alpha,\\
		&
		\text{and either } I_1(y)=I_1(x^k)
		\text{ or } f(x^k)-f(y)\ge\delta_{\rm supp}
	\end{aligned}
	\right\}
\end{equation}
are admissible, where \(\alpha>0\) and \(\delta_{\rm supp}>0\) are prescribed
constants. Thus, every admissible candidate must have full support and all its
active components must be uniformly bounded away from zero. In addition, a
candidate that changes the current support is admissible only if it produces a
fixed objective decrease of at least \(\delta_{\rm supp}\). Same-support candidates
are not required to satisfy this fixed-decrease condition.
The next iterate is then selected according to
\begin{equation}\label{eq:accept-safeguard}
	f(x^{k+1})
	\le
	\min_{z\in\widehat{\Omega}^k} f(z)+\xi^k .
\end{equation}
If \(\widehat{\Omega}^k=\{x^k\}\), then no admissible improving candidate has been
found. Under full-support initialization with
\begin{equation}\label{eq:full-support-alpha-init}
	\|x^0\|_0=s,
	\qquad
	\min_{i\in I_1(x^0)}|x_i^0|\ge\alpha, 
\end{equation}
the safeguarded admissible set remains nonempty along the generated sequence.
This follows by induction and is proved in
Lemma~\ref{lem:eventual-full-support-stabilization}, below. Hence the safeguarded
acceptance rule is well defined.

The purpose of the safeguard is purely analytical. It is used to prove eventual
full-support stabilization. Since every support-changing accepted step must decrease
the objective by at least \(\delta_{\rm supp}\), infinitely many support changes would
force an unbounded total decrease, contradicting lower boundedness of \(f\) on the
level set containing the iterates. Therefore, only finitely many support changes, Lemma~\ref{lem:eventual-full-support-stabilization}, can
occur under the safeguard.

The safeguard is not part of the  stationarity definition. It is an
algorithmic device used only in the local convergence analysis. The tolerance
\(\xi^k\) denotes the outer candidate-selection error, whereas \(\varepsilon_T^k\)
denotes the inner restricted-solver error. In the convergence analysis these errors
are controlled so that the effective candidate-set inexactness satisfies
Assumption~\ref{ass:subproblem} (below).

For a fixed ZCW super support \(T\) generated by the selected active index,
inactive index, sorting permutation, and cutoff index, the relevant optimality
test is related to the restricted subproblem
\begin{equation}\label{eq:ressubpro}
	\min\{f(y):y\in C,\ I_1(y)\subseteq T\}.
\end{equation}
The current point passes this fixed-support ZCW test if

\[
f(x^k)
\le
\min\{f(y):y\in C,\ I_1(y)\subseteq T\}.
\]

Equivalently, the generated support \(T\) reveals a ZCW violation whenever this
restricted subproblem admits a point with objective value strictly below
\(f(x^k)\). Thus, the analysis does not rely on a gradient-score residual, it
relies on whether the restricted ZCW subproblem over the generated support can
produce a genuine objective decrease.

Under the ZCW-aware reservoir injection mechanism, inactive coordinates and their associated ZCW super supports are sampled with uniformly positive probability along full-support subsequences. Consequently, if a full-support accumulation point violates ZCW stationarity, then a violating super support is generated infinitely often almost surely. Provided that the restricted subproblems are solved with errors \(\varepsilon_T^k\to0\) and the
outer candidate-selection errors satisfy Assumption~\ref{ass:subproblem}, such
violations cannot be overlooked indefinitely. This yields
Theorem~\ref{thm:RZCW-PSS-ZCW} in Section~\ref{sec:conv}, which states that every full-support accumulation point generated by {\tt RZCW-PSS} is ZCW stationary almost surely.

\subsection{The {\tt RZCW-PSS} Algorithm}

Algorithm~\ref{alg:RZCW-PSS} is our new algorithm,  called 
Reservoir ZCW Projected Subspace Search ({\tt RZCW-PSS}).  Algorithm \ref{alg:RZCW-PSS} has the following steps. Candidate points generated by sampling and subspace exploration are projected onto
\(\Omega = C\cap C_s\) whenever feasibility is not automatic. In contrast, the
simple-CW and ZCW probes are constructed through symmetry-aware
support modifications and restricted subproblems, and are not interpreted as
projected coordinate perturbations.

In Step~\textbf{S0}, the initial point is projected onto the feasible set
\(C\cap C_s\). A collection of well-separated perturbation directions
\(\{d^k\}_{k=1}^m\) is generated using the {\tt p-usequence} procedure, and the
candidate points
\[
y^k := \proj_{C \cap C_s}(x^0 + d^k), \qquad \forall k\in[m],
\]
are formed. The reservoir \(\mathcal R\subset C\cap C_s\) is initialized with
these projected perturbations using objective-based and distance-based pruning to
enforce diversity. Finally, a {\tt BFS} warm-start procedure is applied with
\[
\{u^k\}_{k=1}^{m+1}=\{y^1,\ldots,y^m, x^0\},\qquad
i_b:=\argmin\{f(u^k)\}_{k=1}^{m+1},
\]
and initialized at \(x:=u^{i_b}\). The output of this warm start serves as the
starting iterate for the main {\tt RZCW-PSS} loop.

In Step~\textbf{S1}, for \(k=0,1,2,\ldots\), the candidate set is initialized as
\(\Omega^k:=\{x^k\}\). Steps~\textbf{S1a}--\textbf{S1f} then generate feasible
candidates, all of which are inserted into \(\Omega^k\). The next iterate
\(x^{k+1}\) is selected from \(\Omega^k\) by the acceptance rule in
Step~\textbf{S1g}.

In Step~\textbf{S1a}, when \(\|x^k\|_0<s\), projected one-dimensional
coordinate-expansion searches are performed along coordinate directions. For
each generated scalar step, the trial point is projected onto
\(\Omega=C\cap C_s\), inserted into the current candidate set \(\Omega^k\), and
used to update the reservoir through the objective- and distance-based pruning
rule.

In Step~\textbf{S1b}, when \(\|x^k\|_0<s\) after Step~\textbf{S1a}, randomized
sparse subspace searches are constructed from the reservoir. The resulting
projected subspace candidates are generated as in
\eqref{eq:projected-subspace-candidate}, with \(x=x^k\), inserted into
\(\Omega^k\), and used to update the reservoir.

In Step~\textbf{S1c}, when \(\|x^k\|_0=s\), projected active-coordinate
refinement candidates are generated for the active coordinates and the
simple-CW swap candidates in \(\mathcal S_{\rm CW}^k\) from
\eqref{eq:simple-cw-candidate-set-generic} are generated. All generated
candidates are inserted into \(\Omega^k\). If the optional {\tt BFS} polishing
step is enabled, the best candidate currently available after this step is used
as a seed for a short {\tt BFS} refinement, and the resulting polished point is
inserted into \(\Omega^k\) as an additional candidate. It is not accepted
immediately.

In Step~\textbf{S1d}, when \(\|x^k\|_0=s\) after Step~\textbf{S1c}, randomized
sparse subspace searches are constructed from the reservoir. The resulting
projected subspace candidates are generated as in
\eqref{eq:projected-subspace-candidate}, with \(x=x^k\), inserted into
\(\Omega^k\), and used to update the reservoir. These subspace moves are used
only as projected continuous exploration steps and are kept distinct from the
support-exchange tests.

In Step~\textbf{S1e}, when \(\|x^k\|_0=s\), the ZCW-aware reservoir injection is
activated with probability \(p^k\in(0,1)\) satisfying \eqref{eq:infpk}. A
randomized {\tt BFS} reservoir injection is then performed as described in
Subsection~\ref{sec:res-cw}. A ZCW super support \(T^k\) is constructed using
the symmetry-aware selection rules explained by \eqref{eq:supconrul}. A refined
point \(r^k\) satisfying \eqref{eq:defxik_final_Tk} is computed by a suitable
restricted inner solver initialized at a warm-start point. Here,
\(\{\varepsilon_T^k\}_{k\ge 1}\) is the sequence of nonnegative inner
restricted-solver errors. The point \(r^k\) is inserted into the current
candidate set and into the reservoir \(\mathcal R\) using objective- and
distance-based pruning. If the optional {\tt BFS} polishing step is enabled,
{\tt BFS} is initialized at \(r^k\), and the polished point is inserted into
\(\Omega^k\) as an additional candidate. Neither \(r^k\) nor its polished
version is accepted immediately. Inserting \(r^k\) into the current candidate set
ensures that any improvement found by the restricted ZCW subproblem is
immediately available for acceptance by the algorithm.

In Step~\textbf{S1f}, if no sufficient decrease has been generated after
Steps~\textbf{S1a}--\textbf{S1e}, a refinement mechanism is activated. First, if
enabled, an optional {\tt BFS}-refined candidate is generated by initializing
{\tt BFS} at the current iterate \(x^k\). This candidate is inserted into
\(\Omega^k\), but is not accepted immediately. In addition, a global refinement
mechanism based on a fresh projected uniform sampling procedure
({\tt p-usequence}) may be activated. A new collection of well-separated
feasible points in \(C\cap C_s\) is generated, evaluated, inserted into
\(\Omega^k\), and used to update the reservoir through the same pruning rule.
If no admissible improving candidate is selected in Step~\textbf{S1g} after a
bounded number of refinement attempts, the algorithm terminates.

In Step~\textbf{S1g}, after Steps~\textbf{S1a}--\textbf{S1f}, the next iterate
\(x^{k+1}\) is selected using either the standard inexact candidate-set rule
\eqref{eq:accept-no-safeguard} over the generated candidate set \(\Omega^k\), or,
when the optional support-identification safeguard is used, the safeguarded rule
\eqref{eq:accept-safeguard} over the filtered admissible set
\(\widehat{\Omega}^k\) defined in \eqref{eq:Omega-hat-safeguard}. The current
iterate is always retained as a fallback candidate. A newly generated point is
accepted only if it produces an actual objective decrease beyond the numerical
acceptance tolerance; otherwise, the algorithm keeps \(x^k\). The safeguarded
variant is used only to enforce eventual full-support stabilization in the local
convergence analysis.

The optional {\tt BFS} polishing steps described above are implementation
enhancements; they only add extra candidates to \(\Omega^k\) and are therefore
not listed as separate algorithmic steps.

\begin{algorithm}[http!]
	\caption{{\tt RZCW-PSS} - Reservoir ZCW Projected Subspace Search}
	\label{alg:RZCW-PSS}
	\begin{algorithmic}
		\State[\textbf{S0}] Project \(x^0\) onto \(C\cap C_s\), initialize the reservoir \(\mathcal R\), and compute the warm-started initial iterate.
		\For{\(k=0,1,2,\ldots\)}
		\State[\textbf{S1}] Initialize \(\Omega^k:=\{x^k\}\), generate candidates through Steps 
		\State ~\textbf{S1a}--\textbf{S1f}, and select \(x^{k+1}\) in Step~\textbf{S1g}.
		\State[\textbf{S1a}] If \(\|x^k\|_0<s\), generate coordinate-expansion candidates 
		\State and insert them into \(\Omega^k\).
		\State[\textbf{S1b}] If \(\|x^k\|_0<s\), generate reservoir-based projected  
		\State subspace-expansion candidates and insert them into \(\Omega^k\).
		\State[\textbf{S1c}] If \(\|x^k\|_0=s\), generate the active-coordinate refinement  
		\State candidate and the simple-CW swap candidates, and insert them \State into \(\Omega^k\).
		\State[\textbf{S1d}] If \(\|x^k\|_0=s\), generate reservoir-based projected
		\State subspace-refinement candidates  and insert them into \(\Omega^k\).
		\State[\textbf{S1e}] If \(\|x^k\|_0=s\), activate the ZCW-aware reservoir injection 
		\State with probability \(p^k\), generate \(r^k\) satisfying \eqref{eq:defxik_final_Tk}, and insert \(r^k\) 
		\State into \(\Omega^k\) and \(\mathcal R\).
		\State[\textbf{S1f}] If no sufficient decrease is obtained, generate optional {\tt BFS}-refined  
		\State and global-refinement candidates, insert them into \(\Omega^k\), and terminate 
		\State if no admissible  improving candidate is found after bounded attempts.
		\State[\textbf{S1g}] Select \(x^{k+1}\) using either \eqref{eq:accept-no-safeguard} over \(\Omega^k\) or \eqref{eq:accept-safeguard} over \(\widehat{\Omega}^k\).
		\EndFor
	\end{algorithmic}
\end{algorithm}

Further implementation details, including tuning parameters, numerical solvers,
subspace construction, and global refinement for {\tt RZCW-PSS}, are provided in
\cite[Sections~14 and 15]{suppMat}.

\section{Convergence, Rates, and Complexity}

This section provides a comprehensive theoretical analysis of the proposed
{\tt RZCW-PSS} algorithm, covering convergence guarantees, local rate of
convergence, and iteration complexity.
The analysis is structured around four main results.
First, Theorem~\ref{thm:locmin-ZCW} establishes that global minimizers satisfy ZCW stationarity, linking global optimality to the stationarity notion adopted in this work. Second, Theorem~\ref{thm:RZCW-PSS-ZCW} proves that every full-support accumulation
point of the algorithm is  ZCW stationary almost surely.  Finally, under
the optional support-identification safeguard and full-support initialization,
Corollary~\ref{cor:all-accumulation-zcw-safeguard} strengthens this conclusion
to all accumulation points.  

\bfi{Detection and acceptance of restricted descent subproblems.}
The analysis relies on the fact that whenever a  ZCW violation is present at a full-support accumulation point, there exists a ZCW super support \(T^\ast\) whose restricted subproblem gives a strict objective decrease. The ZCW-aware reservoir injection mechanism is assumed to generate such violating super supports with uniformly positive probability near
the accumulation point. Together with the restricted-solver accuracy \(\varepsilon_T^k\to0\) and the
outer candidate-selection accuracy in Assumption~\ref{ass:subproblem}, this prevents the algorithm from repeatedly overlooking valid restricted ZCW
descent subproblems.

We discuss complexity for the new algorithm in \cite{suppMat}, in which Proposition~2  quantifies the corresponding conditional local
iteration complexity, yielding $O\!\left(
\log\!\left(\epsilon^{-1}\right)
\right)$ iterations after support stabilization.

\subsection{Convergence Analysis}\label{sec:conv}

In this section, we establish basic convergence properties of the proposed
{\tt RZCW-PSS} algorithm and prove almost sure  ZCW stationarity of full-support
accumulation points. Under the optional support-identification safeguard, this conclusion extends to all accumulation points. 

The proof strategy is organized as follows. Assumptions~\ref{ass:bounded-sublevel}
and~\ref{ass:subproblem} first provide boundedness of the iterates and sufficient
descent through the candidate-set acceptance rule; in particular, the auxiliary
descent result from \cite[Lemma~1]{suppMat} yields convergence of the objective
values and vanishing actual accepted decrease. The probabilistic framework then
defines the events under which the ZCW-aware reservoir injection mechanism
generates a prescribed ZCW super support. Lemma~\ref{lem.Levy} is used to convert
uniform conditional hitting probabilities into almost-sure infinitely-often
generation of such supports.

The deterministic part of the argument begins with
Lemma~\ref{lem:activation-to-swap}, which shows that if a full-support point is
not ZCW stationary, then some ZCW restricted subproblem gives a strict objective
decrease. Lemma~\ref{lem:c0-bfs} ensures that the corresponding ZCW super support
is sampled with uniformly positive probability along subsequences converging to
that point. Proposition~\ref{prop:nondegeneracy},
Lemma~\ref{lem:proj-stability}, and Theorem~\ref{thm:support-identification}
provide the required local support stability near full-support accumulation
points. Theorem~\ref{thm:locmin-ZCW} provides a consistency result for the ZCW
stationarity concept by showing that every global minimizer over \(\Omega\)
is ZCW stationary. Lemma~\ref{lem:uniform-descent} then transfers the strict descent at the limit
point to a uniform descent estimate in a neighborhood. Combining these
ingredients, Theorem~\ref{thm:RZCW-PSS-ZCW} proves that any non-ZCW
full-support accumulation point would generate accepted decreases bounded away
from zero infinitely often, contradicting the vanishing actual accepted decrease.
Finally, Lemma~\ref{lem:eventual-full-support-stabilization} shows that the
optional support-identification safeguard forces eventual full-support
stabilization, and Corollary~\ref{cor:all-accumulation-zcw-safeguard} upgrades
the conclusion from full-support accumulation points to all accumulation points.
The logical dependence among these results is summarized in
Figure~\ref{fig:conv-roadmap}.

To clarify the logical structure of the local convergence-rate analysis,
Figure~\ref{fig:local-rate-roadmap} summarizes the dependencies among the
fixed-support regime, the restricted regularity assumptions, the
candidate/inexactness conditions, and the descent estimate leading to
Theorem~\ref{thm:local-linear}.

\begin{figure}[t]
	\centering
	\scalebox{0.9}{\begin{tikzpicture}[
		node distance=1.8cm and 1.6cm,
		box/.style={
			draw,
			rounded corners,
			align=center,
			text width=3.8cm,
			minimum height=1.05cm,
			font=\small
		},
		arrow/.style={->, thick},
		dashedarrow/.style={->, thick, dashed}
		]
		
		\node[box] (assumptions)
		{{\bf Assumptions}\\
			\textbf{\ref{ass:bounded-sublevel}--\ref{ass:subproblem}:}\\
			\textbf{boundedness + descent}};
		
		\node[box, right=0.5cm of assumptions] (prob)
		{{\bf Probabilistic framework}\\
			\textbf{events \(A^k,T^k,E^k\)}\\
			\textbf{Lemma~\ref{lem.Levy}}};
		
		\node[box, right=0.5cm of prob] (zcwdescent)
		{{\bf ZCW violation}\\
			\textbf{\(\Rightarrow\) restricted descent}\\
			\textbf{Lemma~\ref{lem:activation-to-swap}}\\
			\textbf{Lemma~\ref{lem:c0-bfs}}};
		
		\node[box, below=0.3cm of prob] (support)
		{{\bf Support stability}\\
			\textbf{Proposition~\ref{prop:nondegeneracy}}\\
			\textbf{Lemma~\ref{lem:proj-stability}}\\
			\textbf{Theorem~\ref{thm:support-identification}}};
		
		\node[box, below=0.3cm of support] (uniform)
		{{\bf Uniform local descent}\\
			\textbf{Lemma~\ref{lem:uniform-descent}}};
		
		\node[box, below=0.3cm of uniform] (main)
		{{\bf Almost-sure ZCW stationarity}\\
			\textbf{Theorem~\ref{thm:RZCW-PSS-ZCW}}};
		
		\node[box, right=0.5cm of main] (locmin)
		{{\bf Global minimizers}\\
			\textbf{\(\Rightarrow\) ZCW stationary}\\
			\textbf{Theorem~\ref{thm:locmin-ZCW}}};
		
		\node[box, below=0.3cm of main] (safeguard)
		{{\bf Safeguard upgrade}\\
			\textbf{Lemma~\ref{lem:eventual-full-support-stabilization}}\\
			\textbf{Corollary~\ref{cor:all-accumulation-zcw-safeguard}}};
		
		\draw[arrow] (assumptions) -- (prob);
		\draw[arrow] (prob) -- (zcwdescent);
		\draw[arrow] (zcwdescent) |- (uniform);
		\draw[arrow] (support) -- (uniform);
		\draw[arrow] (uniform) -- (main);
		\draw[arrow] (main) -- (safeguard);
		
		\draw[arrow] (assumptions) |- (support);
		
	\end{tikzpicture}}
	\caption{Logical roadmap of the convergence analysis.}
	\label{fig:conv-roadmap}
\end{figure}
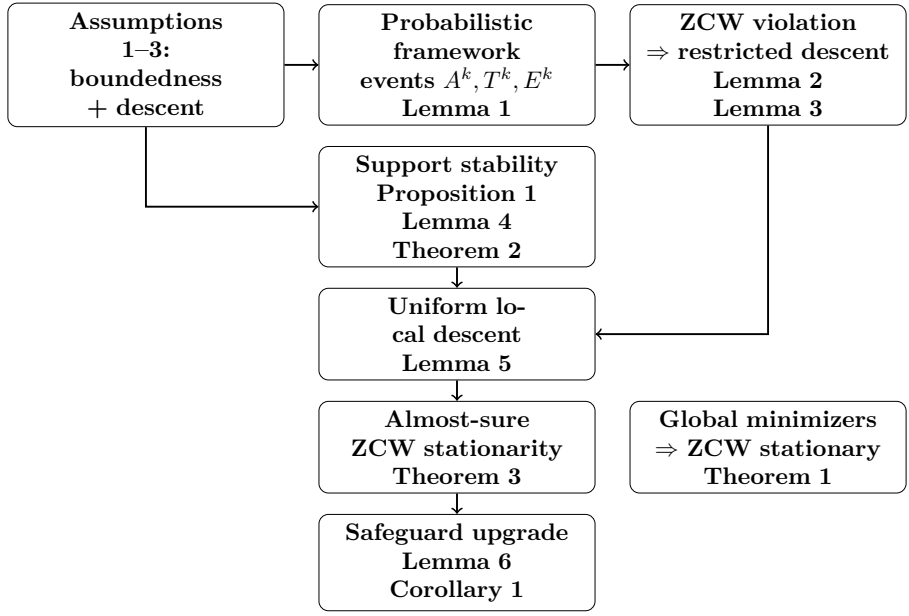

\begin{figure}[t]
	\centering
	\scalebox{0.9}{\begin{tikzpicture}[
		node distance=1.8cm and 1.6cm,
		box/.style={
			draw,
			rounded corners,
			align=center,
			text width=3.8cm,
			minimum height=1.05cm,
			font=\small
		},
		arrow/.style={->, thick}
		]
		
		\node[box] (support)
		{{\bf Fixed-support regime}\\
			\textbf{Step~\textbf{S1g} + full support}\\
			\textbf{Lemma~\ref{lem:eventual-full-support-stabilization}}};
		
		\node[box, right=0.5cm of support] (regularity)
		{{\bf Restricted regularity}\\
			\textbf{Assumptions~\ref{ass:RSC}, \ref{ass:pg-nondeg}}};
		
		\node[box, right=0.5cm of regularity] (candidate)
		{{\bf Candidate and inexactness}\\
			\textbf{Assumption~\ref{ass:candidate-richness}}\\
			\textbf{Assumption~\ref{ass:subproblem}(ii)}};
		
		\node[box, below=0.3cm of regularity] (descent)
		{{\bf Inexact restricted descent}\\
			\textbf{Proposition~\ref{prop:inexact-descent}}};
		
		\node[box, below=0.3cm of descent] (rate)
		{{\bf Local linear rate}\\
			\textbf{Theorem~\ref{thm:local-linear}}};
		
		\draw[arrow] (support) -- (regularity);
		\draw[arrow] (regularity) -- (candidate);
		
		\draw[arrow] (support.south) |- (descent.west);
		\draw[arrow] (regularity.south) -- (descent.north);
		\draw[arrow] (candidate.south) |- (descent.east);
		
		\draw[arrow] (descent) -- (rate);
		
	\end{tikzpicture}}
	\caption{Logical roadmap of the local convergence-rate analysis.}
	\label{fig:local-rate-roadmap}
\end{figure}
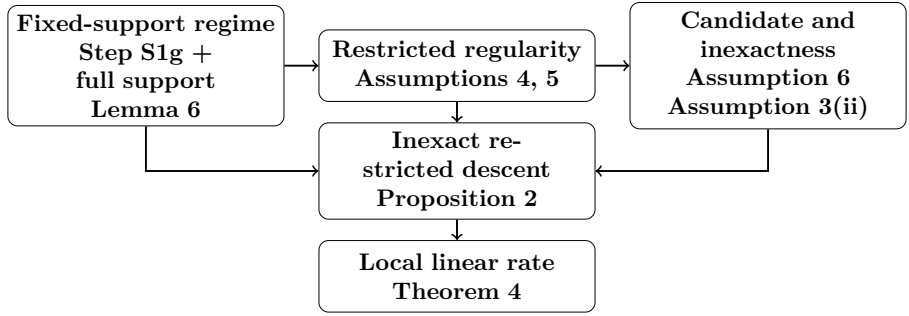

\subsubsection{Problem Setting and Basic Assumptions}
\label{sec:conv:setting}

\begin{ass}\label{ass:bounded-sublevel}
	Let $\Omega = C \cap C_s$. The sublevel set
	\[
	\mathcal L(x^0)
	:=
	\{x \in \Omega : f(x) \le f(x^0)\}
	\]
	is bounded.
\end{ass}

\bfi{Lower boundedness of the objective.}
Since $f$ is continuous and $\Omega$ is closed, the sublevel set
$\mathcal L(x^0)$ is closed, together with boundedness from
Assumption~\ref{ass:bounded-sublevel}, it follows that $\mathcal L(x^0)$ is compact. Therefore, $f$ attains a finite minimum on $\mathcal L(x^0)$, and in particular,
$f$ is bounded below on $\mathcal L(x^0)$.
As all iterates generated by {\tt RZCW-PSS} belong to $\mathcal L(x^0)$,
the objective sequence $\{f(x^k)\}_{k\ge 0}$ is bounded below.

\begin{ass}\label{ass:lip}
	Assume that $f$ is continuously differentiable and that $g$ is Lipschitz
	continuous on bounded subsets of $C$.
\end{ass}

The assumptions imposed in this work are satisfied by a broad class of
constraint sets commonly encountered in sparse optimization, including
the nonnegative orthant, box-constrained sets, probability simplices,
and $\ell_p$-norm balls for $p \ge 1$. These sets are closed and convex,
and satisfy the symmetry properties described in
Subsection~\ref{sec:SymmSet}. Moreover, the projection onto
$\Omega = C \cap C_s$ is well defined for these cases via the sparse
projection operators of \cite{Beck2016}. Therefore, the convergence
analysis applies directly to the standard problem classes considered in
both the literature and our numerical experiments.

We emphasize that the symmetry assumption on $C$ is essential for the theoretical analysis, as it enables explicit characterization of sparse projections and stability of support selection, which are key components of the convergence proofs. 

While the proposed {\tt RZCW-PSS} algorithm can be applied to more general convex constraint sets, the current convergence guarantees rely on this symmetric structure. In particular, for general convex sets with coupling constraints, sparse projections may not admit a tractable structure, necessitating additional analysis. Extending the theoretical results to broader classes of constraint sets without symmetry is an interesting direction for future work.

\begin{ass}[Subproblem accuracy]
	\label{ass:subproblem}
	Let $\Omega^k \subseteq \Omega$ be a candidate set satisfying $x^k \in \Omega^k$,
	and define
	\begin{equation}\label{e.defdeltak}
		\delta^k := f(x^k) - \min_{z \in \Omega^k} f(z) \ge 0.
	\end{equation}
	Assume that the next iterate $x^{k+1}$ satisfies  \eqref{eq:accept-no-safeguard}; that is,
	\[
	f(x^{k+1}) \le \min_{z \in \Omega^k} f(z) + \xi^k,
	\]
	where the error sequence $\{\xi^k\}$ satisfies one of the following:
	\begin{enumerate}
		\item[(i)] $\sum_{k=0}^{\infty} \xi^k < \infty$, or
		\item[(ii)] there exists $\eta \in (0,1)$ such that $\xi^k \le \eta\,\delta^k$ for all $k$.
	\end{enumerate}
\end{ass}

Under either condition (i) or (ii), it will follow from the descent analysis
in Section 2 of \cite{suppMat} that $\delta^k \to 0$. Thus, in both cases,
this further implies that $\xi^k \to 0$.

We define the actual decrease
\begin{equation}\label{e.defDeltak}
	\Delta^k = f(x^k) - f(x^{k+1}).
\end{equation}
The subproblem accuracy assumption in \eqref{e.defdeltak} ensures that approximate solutions in \eqref{eq:accept-no-safeguard}
preserve a sufficient descent condition 
\begin{equation}\label{eq:actdec}
	\Delta^k \ge \delta^k - \xi^k,
\end{equation}
which is central to the convergence analysis.

\subsubsection{Probabilistic Framework and Event Construction}
\label{sec:conv:prob}
Considering probabilistic model, described in Section \ref{sec:res-cw}, we introduce a family of events associated with a reference point
$x^\ast \in \Omega$ with $\|x^\ast\|_0=s$. Let \(T^\ast\) denote a ZCW
super support associated with a violation of  ZCW stationarity
at \(x^\ast\). 
We define the sequence of events $\{E^k(x^\ast)\}_{k\ge 1}$ by
\begin{equation}\label{eq:Ek-def}
	E^k(x^\ast)
	:=
	A^k \cap \{T^k = T^\ast\},
\end{equation}
where $A^k$ denotes the activation of the ZCW-aware reservoir injection mechanism at iteration
$k$, and \(T^k\) is the ZCW super support generated by this injection mechanism.
By construction, $E^k(x^\ast) \in \mathcal F^k$ for all $k \ge 1$.
For notational simplicity, we write $E^k := E^k(x^\ast)$ when the reference point
$x^\ast$ is fixed.

Equivalently, when the active index, inactive index, sorting permutation, and
cutoff defining \(T^k\) are sampled explicitly, the event may be written as
\[
E^k(x^\ast)
=
A^k
\cap
\{q^k=q^\ast\}
\cap
\{j^k=j^\ast\}
\cap
\{T^k=T^\ast\}.
\]
The following result is Lévy's conditional Borel-Cantelli lemma \cite[Section~12.15]{williams1991probability}, which will be used in the sequel. Informally, it asserts that if the cumulative conditional probabilities of a sequence of events, given the past history, diverge almost surely (a.s.), then infinitely many of those events occur (i.o.), almost surely.

\begin{lem}[L\'evy's conditional Borel-Cantelli]\label{lem.Levy}
	Let $(\mathcal F^k)$ be a filtration and $\{E^k\}_{k\geq 1}$ a sequence of events
	with $E^k \in \mathcal F^k$.
	If
	\[
	\sum_{k=1}^\infty \Pr(E^k \mid \mathcal F^{k-1}) = \infty
	\quad \text{a.s.},
	\]
	then $\Pr(E^k \text{ i.o.}) = 1$.
\end{lem}

\subsubsection{Detection of Stationarity Violations}
\label{sec:conv:detection}

This subsection establishes that violations of  ZCW stationarity
are detected with uniformly positive probability.  The key point is that ZCW
stationarity is a support-based condition. Therefore, a violation is not encoded
by a projected one-dimensional perturbation, but by the existence of a
ZCW super support on which the restricted sparse subproblem attains a
strictly smaller objective value.

Consequently, the probabilistic event used in the convergence analysis must
ensure that the ZCW-aware reservoir injection mechanism generates the relevant
ZCW super support. Once this super support is generated, the restricted
subproblem produces a candidate whose objective value is strictly below
that of the current iterate, up to the prescribed inner accuracy.

\begin{lem}[Activation implies restricted ZCW descent]
	\label{lem:activation-to-swap}
	Let $C \subseteq \mathbb{R}^n$ be closed, convex, and either nonnegative
	type-1 symmetric or type-2 symmetric. Moreover, let $x^\ast\in\Omega=C\cap C_s$ satisfying $\|x^\ast\|_0=s$. Suppose that $x^\ast$ violates  ZCW stationarity. Let
	$q^\ast$, $j^\ast$, and \(T^\ast\) denote the indices and ZCW super support
	associated with this violation, where
	\[
	T^\ast
	:=
	\left(
	S^{\pi^\ast}_{[k^\ast,n]}
	\cup I_1(x^\ast)
	\cup \{j^\ast\}
	\right)
	\setminus\{q^\ast\}
	\]
	for some $\pi^\ast\in
	\widetilde{\mathfrak S}
	\left(
	-\mathbf p(-g(x^\ast))
	\right)$ and some \(k^\ast\in[n]\) chosen so that \(|T^\ast|=s\). Then there exists
	\(\delta>0\) such that
	\begin{equation}\label{e.decfdelta}
		\min_{\substack{z \in C \\ \mathrm{supp}(z)\subseteq T^\ast}}
		f(z)
		\;\le\;
		f(x^\ast)-\delta .
	\end{equation}
\end{lem}

\begin{proof}
	Since \(x^\ast\) violates ZCW stationarity, the restricted
	ZCW subproblem associated with the super support \(T^\ast\) strictly improves
	the objective value at \(x^\ast\). Therefore,
	\[
	\min_{\substack{z \in C \\ \mathrm{supp}(z)\subseteq T^\ast}}
	f(z)
	<
	f(x^\ast).
	\]
	Defining
	\[
	\gamma
	:=
	f(x^\ast)
	-
	\min_{\substack{z \in C \\ \mathrm{supp}(z)\subseteq T^\ast}}
	f(z)
	>0
	\]
	and
	choosing \(\delta:=\gamma/2\), we obtain
	\[
	\min_{\substack{z \in C \\ \mathrm{supp}(z)\subseteq T^\ast}}
	f(z)
	=
	f(x^\ast)-\gamma
	=
	f(x^\ast)-2\delta
	\le
	f(x^\ast)-\delta,
	\]
	which proves \eqref{e.decfdelta}.

\end{proof}

The preceding lemma is intentionally simple and a direct consequence of the
definition of ZCW stationarity. Its role is to isolate the deterministic
descent gap created by a ZCW violation. The next lemma concerns the
probabilistic mechanism that guarantees that the corresponding super support is
sampled with uniformly positive probability.

\begin{lem}[Local support-hitting bound along convergent subsequences]
	\label{lem:c0-bfs}
	Let \(x^\ast\in\Omega\) satisfy \(\|x^\ast\|_0=s\), and suppose that
	\(x^\ast\) violates  ZCW stationarity. Let \(T^\ast\) be a ZCW super support associated with this violation. For each iteration \(k\),
	let recall and reform $A^k$ from \eqref{e:ak}:  $A^k = \{\text{\textbf{S1e} is activated at iteration } k\}$. Assume that there exist a neighborhood \(\mathcal N(x^\ast)\) of \(x^\ast\) and a constant \(c_0>0\) such that, whenever \(x^k\in\mathcal N(x^\ast)\),
	the ZCW-aware reservoir injection mechanism satisfies supporting-hitting condition in \eqref{eq:support-hitting}, namely,
	\[
	\Pr(T^k=T^\ast \mid \mathcal F^{k-1},A^k)
	\ge c_0 .
	\]
	Then, for any subsequence \(\{x^{k_\ell}\}_{{k_\ell}\geq 0}\) such that
	\(x^{k_\ell}\to x^\ast\), we have, for all sufficiently large \(\ell\),
	\[
	\Pr(T^{k_\ell}=T^\ast \mid \mathcal F^{k_\ell-1},A^{k_\ell})
	\ge c_0 .
	\]
\end{lem}

\begin{proof}
	Let \(\{x^{k_\ell}\}_{{k_\ell}\geq 0}\) be any subsequence such that
	\(x^{k_\ell}\to x^\ast\). Since \(\mathcal N(x^\ast)\) is a neighborhood of
	\(x^\ast\), there exists \(\ell_0\in\mathbb N\) such that $x^{k_\ell}\in\mathcal N(x^\ast)$ for all $\ell\ge \ell_0$. By the assumed uniform support-hitting condition, every iterate satisfying
	\(x^k\in\mathcal N(x^\ast)\) also satisfies
	\[
	\Pr(T^k=T^\ast \mid \mathcal F^{k-1},A^k)
	\ge c_0 .
	\]
	Applying this statement with \(k=k_\ell\), we obtain, for all
	\(\ell\ge\ell_0\),
	\[
	\Pr(T^{k_\ell}=T^\ast \mid \mathcal F^{k_\ell-1},A^{k_\ell})
	\ge c_0 .
	\]
	This proves the claim.
\end{proof}

The assumption above is a local uniform exploration condition. It states that,
near a nonzero-CW point \(x^\ast\), the ZCW-aware reservoir injection mechanism assigns
positive probability to the ZCW super support \(T^\ast\) responsible for the
violation. The lemma simply records that this local probability bound is inherited
along every subsequence converging to \(x^\ast\).

\subsubsection{Structural Properties of Accumulation Points}
\label{sec:conv:structure}

This subsection records structural properties of full-support accumulation points
of the sequence generated by {\tt RZCW-PSS}.  First, along any subsequence
converging to a full-support point, the active components are eventually bounded
away from zero.  Second, under the symmetry assumptions on \(C\), sparse
projections admit a locally stable support selection near such points.  These
facts are used only as local support-stability tools. These should not be confused
with the  ZCW condition, which is a support-based optimality
condition involving restricted sparse subproblems.

\begin{prop}[Uniform support separation along a convergent subsequence]
	\label{prop:nondegeneracy}
	Let $\hat x \in \Omega$ be an accumulation point of $\{x^k\}_{k\geq 0}$ with
	$\|\hat x\|_0 = s$ and support $T^\ast = I_1(\hat x)$. Then for any subsequence
	$\{x^{k_\ell}\}_{{k_\ell\geq 0}}$ such that $x^{k_\ell} \to \hat x$, there exist
	$\epsilon > 0$ and $\ell_0 \in \mathbb{N}$, depending on the subsequence, such
	that $\D\min_{i \in T^\ast} |x^{k_\ell}_i| \ge \epsilon$ for all $\ell \ge \ell_0$.
\end{prop}

\begin{proof}
	Since $x^{k_\ell} \to \hat x$, for each $i \in T^\ast$, we have $x^{k_\ell}_i \to \hat x_i \neq 0$. Hence, for each $i \in T^\ast$, there exist $\epsilon_i > 0$ and $\ell_i$ such that $|x^{k_\ell}_i| \ge \epsilon_i$ for all $\ell \ge \ell_i$. Since $T^\ast$ is finite, by defining $\epsilon := \D\min_{i \in T^\ast} \epsilon_i > 0$ and $\ell_0 := \D\max_{i \in T^\ast} \ell_i$, we obtain, for all $\ell \ge \ell_0$, $\D\min_{i \in T^\ast} |x^{k_\ell}_i| \ge \epsilon$.

\end{proof}

The following result states that, near a fully sparse point \(\hat x\), one may
select a sparse projection whose super support is the fixed support
\(T=I_1(\hat x)\).  Locally, this selected projection is represented by a standard
convex projection onto the restricted set \(C_T\), see \cite[Section 3]{suppMat}.

\begin{lem}[Local stability of sparse projections under symmetric structure]
	\label{lem:proj-stability}
	Let $C \subseteq \mathbb{R}^n$ be closed, convex, and either nonnegative type-1
	symmetric or type-2 symmetric, let $\hat x \in \Omega=C \cap C_s$ satisfy $\|\hat x\|_0 = s$, and denote its
	support by $T = I_1(\hat x)$. Then there exists $\delta > 0$ such that for all $x \in \mathbb{R}^n$ with
	$\|x-\hat x\|<\delta$, there exists a projection
	$z \in \proj_\Omega(x)$ admitting \(T\) as a super support and satisfying
	\[
	z = U_T \proj_{C_T}(x_T),
	\]
	where $U_T$ is the canonical embedding from $\mathbb{R}^{s}$ into
	$\mathbb{R}^n$ and
	\[
	C_T := \{w \in \mathbb{R}^{s} : U_T w \in C\}.
	\]
	Moreover, such a selection can be made so that $z \to \hat x$ as $x \to \hat x$. In particular, for all \(x\) sufficiently close to \(\hat x\), this selected
	projection satisfies $\operatorname{supp}(z)=T$.
\end{lem}

\begin{proof}
	We have \(T=I_1(\hat x)\), with \(|T|=s\).
	
	\textbf{(i) Support separation.}
	Since \(\hat x_i\neq0\) for all \(i\in T\), we set $\epsilon
	=
	\frac12\min_{i\in T}|\hat x_i|>0$. Then $|\hat x_i|\ge 2\epsilon$ for all $i\in T$ and $\hat x_j=0$ for all $j\notin T$.
	
	\textbf{(ii) Stability of the ordering.}
	Let \(x\in\mathbb R^n\) satisfy \(\|x-\hat x\|_\infty<\epsilon\). Then
	\[
	|x_i|\ge \epsilon
	\quad \forall i\in T,
	\qquad
	|x_j|<\epsilon
	\quad \forall j\notin T.
	\]
	So, for \(x\) sufficiently close to \(\hat x\), in the type-2 symmetric case, the indices in \(T\) are precisely the
	indices of the \(s\) largest absolute components of \(x\) and
	in the nonnegative type-1 case 
	the indices in \(T\) are precisely the indices of the \(s\) largest components
	of \(x\).  Hence, in either symmetry class, \(T\) is selected by the corresponding
	ordering rule used in the sparse projection characterization.
	Therefore, by the sparse projection structure for nonnegative type-1 and
	type-2 symmetric sets, there exists a projection \(z\in\proj_\Omega(x)\)
	admitting \(T\) as a super support and satisfying, see Section ~3 of \cite{suppMat}, 
	\begin{equation}\label{eq:selproj}
		z
		=
		U_T \proj_{C_T}(x_T).   
	\end{equation}
	
	\textbf{(iii) Continuity of the selected projection.}
	Since \(C_T\) is closed and convex, the projection operator
	\(\proj_{C_T}\) is nonexpansive and hence continuous. Therefore,
	\[
	\proj_{C_T}(x_T)
	\to
	\proj_{C_T}(\hat x_T)
	\quad \text{as } x\to\hat x.
	\]
	Because \(\hat x\in C\) and \(\operatorname{supp}(\hat x)=T\), we have
	\(\hat x_T\in C_T\), so that $\proj_{C_T}(\hat x_T)=\hat x_T$. Thus, for the selected projection in \eqref{eq:selproj}, 
	we obtain $z\to U_T\hat x_T=\hat x$. Since every component of \(\hat x_T\) is nonzero and \(z\to\hat x\), it follows
	that, for all \(x\) sufficiently close to \(\hat x\), $\operatorname{supp}(z)=T$. Finally, since \(\|x-\hat x\|_\infty\le \|x-\hat x\|\), choosing \(\delta>0\)
	sufficiently small gives the desired result.

\end{proof}

\begin{rem}[Global selection vs.\ local stability]
	\label{rem:global-vs-local}
	The sparse projection results for symmetric sets provide a global characterization
	of admissible supports. Depending on the symmetry class, a projection can be
	constructed from supports determined by the ordering of \(x\) or \(|x|\).  In
	contrast, Lemma~\ref{lem:proj-stability} establishes a local stability property.
	Around a point \(\hat x\) with full support \(|I_1(\hat x)|=s\), one may select a
	projection whose super support remains fixed and equal to \(I_1(\hat x)\) under
	sufficiently small perturbations of the input.  Moreover, this selected projection
	reduces locally to a standard convex projection onto \(C_T\).
\end{rem}

\begin{thm}[Global minimizers are zero-coordinatewise stationary]
	\label{thm:locmin-ZCW}
	Let \(C\subseteq\mathbb{R}^n\) be closed, convex, and either nonnegative
	type-1 symmetric or type-2 symmetric, and let
	\(\Omega=C\cap C_s\). If \(\hat x\in\Omega\) is a global minimizer of \(f\)
	over \(\Omega\), then \(\hat x\) is a  ZCW stationary point
	of~\eqref{eq:prodef}.
\end{thm}

\begin{proof}
	Since \(\hat x\) is a global minimizer of \(f\) over
	\(\Omega=C\cap C_s\), we have $f(\hat x)\le f(y)$ for all $y\in \Omega$. We first show that \(\hat x\) is a basic feasible point. Let
	\(\mathcal L\subseteq[n]\) be any super support of \(\hat x\), that is, $I_1(\hat x)\subseteq \mathcal L$ and $|\mathcal L|=s$. Every point \(y\in C\) satisfying \(I_1(y)\subseteq \mathcal L\) belongs to
	\(\Omega\), because \(|\mathcal L|=s\). Hence,
	\[
	f(\hat x)\le f(y),
	\qquad
	\forall y\in C
	\quad\text{with}\quad
	I_1(y)\subseteq \mathcal L .
	\]
	Since \(I_1(\hat x)\subseteq\mathcal L\), then
	\(\hat x=U_{\mathcal L}\hat x_{\mathcal L}\), for a proper canonical embedding. Therefore,
	\[
	f(U_{\mathcal L}\hat x_{\mathcal L})
	\le
	f(U_{\mathcal L}w),
	\qquad
	\forall w\in C_{\mathcal L}.
	\]
	Thus, \(\hat x_{\mathcal L}\) is a global minimizer of the restricted smooth
	problem
	\[
	\min\{f(U_{\mathcal L}w): w\in C_{\mathcal L}\}.
	\]
	Since \(C_{\mathcal L}\) is closed and convex, the standard first-order necessary
	condition for smooth optimization over a convex set gives $g_{\mathcal L}(\hat x)^{\top}
	(w-\hat x_{\mathcal L})\ge 0$ for all $w\in C_{\mathcal L}$. Equivalently, for every \(L>0\), $\hat x_{\mathcal L}
	=
	\proj_{C_{\mathcal L}}
	\left(
	\hat x_{\mathcal L}
	-
	(1/L) g_{\mathcal L}(\hat x)
	\right)$. Since \(\mathcal L\) was arbitrary, \(\hat x\) is a basic feasible point, see \cite{Beck2016} and \cite[Section 4]{suppMat}.
	
	It remains to verify the ZCW condition. Let \(T\) be any ZCW
	super support generated by the  ZCW construction at
	\(\hat x\). Then every \(y\in C\) satisfying $I_1(y)\subseteq T$ belongs to \(\Omega\), because \(|T|=s\). Therefore, by global
	optimality of \(\hat x\),
	\[
	f(\hat x)\le f(y),
	\qquad
	\forall y\in C
	\quad\text{with}\quad
	I_1(y)\subseteq T .
	\]
	Taking the minimum over this restricted feasible set gives
	\[
	f(\hat x)
	\le
	\min\left\{
	f(y):
	y\in C,\;
	I_1(y)\subseteq T
	\right\}.
	\]
	Thus, no ZCW restricted super support subproblem improves \(\hat x\).
	Since \(\hat x\) is also basic feasible, \(\hat x\) is ZCW stationary.

\end{proof}

\subsubsection{Global Convergence Analysis}
\label{sec:conv:global}

\bfi{Consistency between algorithm and analysis:} The convergence analysis is built upon the sufficient descent property established in \cite[Lemma~1]{suppMat}, which is enforced through the acceptance rule. At each iteration, the next iterate is selected from a
finite candidate set $\Omega^k$ and achieves near-minimal objective value
within this set. The candidate set $\Omega^k$ is generated by the algorithmic components
in Steps~\textbf{S1a}--\textbf{S1e}. Projected subspace
steps and coordinates provide algorithmic descent candidates, while the simple-CW and
ZCW mechanisms are constructed according to the 
support-based conditions. In particular, Step~\textbf{S1e} generates
restricted sparse subproblems over ZCW super supports.

The randomized reservoir injection step (Step~\textbf{S1e}) is essential for the convergence analysis. It assigns a uniformly positive probability to the generation of every super support associated with a violation of  ZCW stationarity. As a result, each corresponding restricted descent subproblem is selected infinitely often almost surely, by an application of L'evy's conditional Borel-Cantelli lemma (Lemma~\ref{lem.Levy}).

Moreover, the acceptance criterion ensures that whenever a sampled subproblem admits a sufficiently improving solution, the resulting candidate point is accepted, up to the prescribed subproblem accuracy. Together, these properties guarantee that directions revealing violations of  ZCW stationarity cannot be neglected indefinitely.

Additional implementation components, such as approximate subproblems,
reservoir pruning, and warm-start strategies, only affect how candidate points
are generated and do not alter the admissible restricted descent subproblems.
Therefore, they do not affect the validity of the convergence analysis.

The following lemma shows that if a fixed support $T$ is strictly suboptimal at
$x^\ast$ by a margin $2\delta$, then this suboptimality persists uniformly in a
neighborhood of $x^\ast$, up to a loss of at most $\delta$ due to continuity.

\begin{lem}[Uniform descent on a fixed support]
	\label{lem:uniform-descent}
	Let $T \subseteq [n]$ and define
	\[
	\phi := \min \{ f(z) : z \in C,\ I_1(z) \subseteq T \}.
	\]
	Suppose that for some $x^\ast \in \Omega$, $\phi \le f(x^\ast) - 2\delta$ for some $\delta > 0$.
	Then there exists a neighborhood $\mathcal N$ of $x^\ast$ such that for all
	$x \in \mathcal N \cap \Omega$, $\phi \le f(x) - \delta$.
\end{lem} 

\begin{proof}
	By continuity of $f$ at $x^\ast$, for $\varepsilon = \delta$ there exists a
	neighborhood $\mathcal N$ of $x^\ast$ such that for all $x \in \mathcal N$, $|f(x) - f(x^\ast)| < \delta$, which implies $f(x) \ge f(x^\ast) - \delta$. Combining this with the assumption $\phi \le f(x^\ast) - 2\delta$, we obtain $\phi \le f(x^\ast) - 2\delta \le f(x) - \delta$, which proves the claim. 
\end{proof}

\begin{thm}[Subsequential support stability]
	\label{thm:support-identification}
	Let $\{x^k\}_{k\geq 0}$ be the sequence generated by {\tt RZCW-PSS}, and let
	$x^\ast$ be an accumulation point satisfying $\|x^\ast\|_0=s$. If
	$\{x^{k_\ell}\}_{{k_\ell}\geq 0}$ is any subsequence such that $x^{k_\ell}\to x^\ast$, then
	there exists $\ell_0\in\mathbb N$ such that
	\begin{equation}\label{e.fixedsupport}
		I_1(x^{k_\ell})=I_1(x^\ast),
		\qquad
		\forall \ell\ge \ell_0 .
	\end{equation}
\end{thm}

\begin{proof}
	Let $T^\ast:=I_1(x^\ast)$. Since $\|x^\ast\|_0=s$, we have $|T^\ast|=s$ and
	$x^\ast_i\neq0$ for all $i\in T^\ast$.
	
	By Proposition~\ref{prop:nondegeneracy}, there exist $\epsilon>0$ and
	$\ell_0\in\mathbb N$ such that, for all $\ell\ge\ell_0$, $\min_{i\in T^\ast}|x^{k_\ell}_i|\ge \epsilon$. Hence, $T^\ast\subseteq I_1(x^{k_\ell})$ for all $\ell\ge\ell_0$. Since each iterate satisfies $\|x^{k_\ell}\|_0\le s$ and $|T^\ast|=s$, it follows that $I_1(x^{k_\ell})=T^\ast=I_1(x^\ast)$ for all $\ell\ge\ell_0$.

\end{proof}

\begin{thm}
	\textbf{(Almost sure ZCW stationarity of {\tt RZCW-PSS}).}
	\label{thm:RZCW-PSS-ZCW}
	
	Assume that Assumptions~\ref{ass:bounded-sublevel}--\ref{ass:subproblem} hold.
	In particular, the iterates satisfy the inexactness condition
	\eqref{eq:accept-no-safeguard}. Moreover, suppose that the {\tt RZCW-PSS}
	algorithm is implemented so that, at every iteration $k$ with $\|x^k\|_0=s$,
	the ZCW-aware reservoir injection mechanism~\textbf{S1e} is activated with probability
	$p^k \in (0,1)$ satisfying \eqref{eq:infpk}, and \eqref{eq:pklb}.
	
	Assume further that, for every full-support accumulation point $x^\ast$ that is
	not  ZCW stationary, and for every ZCW super support
	$T^\ast$ associated with a violation at $x^\ast$, the injection mechanism
	satisfies the uniform support-hitting condition
	\[
	\Pr(T^k=T^\ast \mid \mathcal F^{k-1}, A^k)
	\ge c_0
	\]
	for some constant $c_0>0$, whenever $x^k$ is sufficiently close to $x^\ast$. Assume also that the restricted ZCW subproblem solvability condition stated in
	Subsection~\ref{sec:res-cw} holds, and that the restricted ZCW subproblem
	errors in \eqref{eq:defxik_final_Tk} satisfy \(\varepsilon_T^k\to0\).
	
	Then every full-support accumulation point of the iterates generated by
	{\tt RZCW-PSS} is a  ZCW stationary point almost surely.
\end{thm}

\begin{proof}
	By construction, all {\tt RZCW-PSS} iterates belong to the feasible set
	\(\Omega=C\cap C_s\).
	
	\textbf{(i) Convergence of the objective and vanishing actual change.}
	By \cite[Lemma 1]{suppMat}, the sequence $\{f(x^k)\}_{k\ge0}$ converges and
	the actual change in \eqref{e.defDeltak}
	satisfies $\Delta^k\to0$.
	In particular, no update yielding a fixed positive decrease can occur
	infinitely often.
	
	\textbf{(ii) Boundedness and accumulation points.}
	By the monotone acceptance rule, \(f(x^k)\le f(x^0)\) for all \(k\). Hence all
	accepted iterates remain in the initial sublevel set $\mathcal L(x^0)$.
	Assumption~\ref{ass:bounded-sublevel} implies that the sequence $\{x^k\}_{k\geq 0}$ is
	bounded. Hence, it admits accumulation points.
	
	Let \(\Xi_0\) denote the probability-one event on which the conclusions of
	\cite[Lemma 1]{suppMat}, mentioned in \bfi{(i)}, hold. In particular, for every sample path in
	\(\Xi_0\), $\Delta^k\to0$. We fix an arbitrary sample path in \(\Xi_0\). The conditional
	Borel-Cantelli argument in Lemma~1 will be applied below, after the events \(F^k\) have
	been defined and after the divergence of the corresponding conditional
	probability sum has been verified.
	
	Let \(x^\ast\) be an arbitrary full-support accumulation point on this sample
	path, and let \(\{x^{k_\ell}\}_{k_\ell\geq 0}\) be a subsequence such that $x^{k_\ell}\to x^\ast$ with $\|x^\ast\|_0=s$. By Theorem~\ref{thm:support-identification}, after passing to the tail of the
	subsequence if necessary, we have $I_1(x^{k_\ell})=I_1(x^\ast)$ for all sufficiently large $\ell$. In particular, $\|x^{k_\ell}\|_0=s$ for all sufficiently large $\ell$.
	
	\textbf{(iii) Contradiction hypothesis.}
	Assume, by contradiction, that $x^\ast$ is not  ZCW stationary.
	Then, by Lemma~\ref{lem:activation-to-swap}, there exists a ZCW
	super support $T^\ast$ and a constant $\delta_0>0$ such that
	\[
	\min_{\substack{z\in C\\ I_1(z)\subseteq T^\ast}}
	f(z)
	\le
	f(x^\ast)-2\delta_0 .
	\]
	Applying Lemma~\ref{lem:uniform-descent}, there exists a neighborhood
	$\mathcal N$ of $x^\ast$ such that for all $x\in\mathcal N\cap\Omega$,
	\begin{equation}\label{e:unidec}
		\min_{\substack{z\in C\\ I_1(z)\subseteq T^\ast}}
		f(z)
		\le
		f(x)-\delta_0.
	\end{equation}
	Since \(x^{k_\ell}\to x^\ast\), there exists \(\ell_0\) such that $x^{k_\ell}\in\mathcal N$ and  $\|x^{k_\ell}\|_0=s$ for all $\ell\ge\ell_0$. Thus, along the relevant subsequence, the ZCW-aware reservoir injection mechanism is
	eligible to be activated at all sufficiently large iterates.
	
	\textbf{(iv) Almost-sure detection of the violating super support.}
	We use the convention that \(\mathcal F^{k-1}\) represents all information
	available immediately before the fresh random choices of iteration \(k\) are
	made. In particular, the current iterate \(x^k\) has already been computed from
	the previous iterations, and hence \(x^k\) is \(\mathcal F^{k-1}\)-measurable. Thus, any event determined only by the current iterate is already known at time
	\(\mathcal F^{k-1}\). Therefore, for any Borel set \(B\subseteq\mathbb R^n\),
	\[
	(x^k)^{-1}(B)
	=
	\{\omega : x^k(\omega)\in B\}
	\in\mathcal F^{k-1}.
	\]
	Taking \(B=\mathcal N\), where \(\mathcal N\) is a fixed neighborhood of \(x^\ast\), yields
	\[
	\{x^k\in\mathcal N\}
	=
	(x^k)^{-1}(\mathcal N)\in\mathcal F^{k-1}.
	\] Similarly, since $\{x\in\mathbb R^n:\|x\|_0=s\}$ is a Borel set, the following event is also \(\mathcal F^{k-1}\)-measurable,
	\[
	\{\|x^k\|_0=s\}
	=
	(x^k)^{-1}\bigl(\{x\in\mathbb R^n:\|x\|_0=s\}\bigr)\in\mathcal F^{k-1}.
	\]
	Hence the event
	\(\{x^k\in\mathcal N\}\cap\{\|x^k\|_0=s\}\) is
	\(\mathcal F^{k-1}\)-measurable. 
	Define the neighborhood hit event
	\[
	F^k :=
	\{x^k\in\mathcal N\}
	\cap
	\{\|x^k\|_0=s\}
	\cap A^k
	\cap \{T^k=T^\ast\}.
	\]
	Let $B^k:=\{x^k\in\mathcal N,\ \|x^k\|_0=s\}$. Since \(B^k\in\mathcal F^{k-1}\), we have
	\[
	\Pr(F^k\mid\mathcal F^{k-1})
	=
	\mathbf 1_{B^k}
	\Pr\!\bigl(A^k\cap\{T^k=T^\ast\}\mid\mathcal F^{k-1}\bigr).
	\]
	By the conditional chain rule, the activation lower bound \eqref{eq:pklb}, and the uniform
	support-hitting condition \eqref{eq:support-hitting}, whenever \(x^k\in\mathcal N\) and
	\(\|x^k\|_0=s\),
	\[
	\Pr\!\bigl(A^k\cap\{T^k=T^\ast\}\mid\mathcal F^{k-1}\bigr)
	=
	\Pr(A^k\mid\mathcal F^{k-1})
	\Pr(T^k=T^\ast\mid\mathcal F^{k-1},A^k)
	\ge
	p_{\mathrm{inj}}c_0.
	\]
	Therefore,
	\begin{equation}\label{eq:prob-factorization}
		\Pr(F^k\mid\mathcal F^{k-1})
		\ge
		\mathbf 1_{B^k}\,
		p_{\mathrm{inj}}c_0 .
	\end{equation}
	Since \(x^\ast\) is a full-support accumulation point and
	\(x^{k_\ell}\to x^\ast\), the event \(B^k\) occurs infinitely often along the
	chosen sample path. Hence, $\sum_{k=1}^{\infty}
	\mathbf 1_{B^k}
	=
	\infty$. So, by \eqref{eq:prob-factorization},
	\[
	\sum_{k=1}^{\infty}
	\Pr(F^k\mid\mathcal F^{k-1})
	\ge
	p_{\mathrm{inj}}c_0
	\sum_{k=1}^{\infty}
	\mathbf 1_{B^k}
	=
	\infty .
	\]
	By L\'evy's conditional Borel-Cantelli lemma, the events \(F^k\) occur
	infinitely often almost surely for the event sequence associated with the fixed
	full-support accumulation point \(x^\ast\) and its violating ZCW
	super support \(T^\ast\).
	
	\textbf{(v) Contradiction with vanishing actual change.}
	Whenever \(F^k\) occurs, we have \(x^k\in\mathcal N\), \(\|x^k\|_0=s\), and
	\(T^k=T^\ast\). Therefore, the ZCW-aware reservoir injection mechanism generates a probe candidate
	\(r^k\) satisfying \eqref{eq:defxik_final_Tk} with $T^\ast$
	\[
	f(r^k)
	\le
	\min_{\substack{z\in C\\ I_1(z)\subseteq T^\ast}}
	f(z)
	+
	\varepsilon_T^k .
	\]
	Since \(x^k\in\mathcal N\), the uniform descent estimate in \eqref{e:unidec} gives $f(r^k)
	\le
	f(x^k)-\delta_0+\varepsilon_T^k$. The point \(r^k\) belongs to the current candidate set \(\Omega^k\). Hence,
	using the candidate-set acceptance condition in \eqref{eq:accept-no-safeguard}, we obtain
	\[
	f(x^{k+1})
	\le
	f(r^k)+\xi^k
	\le
	f(x^k)-\delta_0+\varepsilon_T^k+\xi^k .
	\]
	Since \(\xi^k\to0\) by \cite[Lemma 1]{suppMat} and
	\(\varepsilon_T^k\to0\) by assumption, for all sufficiently large \(k\) such that
	\(F^k\) occurs, $\varepsilon_T^k+\xi^k\le \delta_0/2$. Therefore, $f(x^{k+1})
	\le
	f(x^k)-\delta_0/2$. Equivalently by \eqref{e.defDeltak}, $\Delta^k
	\ge \delta_0/2$. Since the events \(F^k\) occur infinitely often almost surely on the sample paths under consideration, this implies $\limsup_{k\to\infty}\Delta^k
	\ge \delta_0/2
	>0$, which contradicts \(\Delta^k\to0\).
	
	Therefore, the contradiction hypothesis is false. Since \(x^\ast\) was arbitrary
	among full-support accumulation points on the sample path under consideration,
	every full-support accumulation point is  ZCW stationary almost
	surely.

\end{proof}

\begin{lem}[Eventual full-support stabilization under the safeguard]
	\label{lem:eventual-full-support-stabilization}
	Assume that the iterates generated by {\tt RZCW-PSS} remain in
	\(\mathcal L(x^0)\), that \(f\) is bounded below on \(\mathcal L(x^0)\), and
	that the accepted objective values are nonincreasing. Suppose that the optional
	support-identification safeguard in Step~\textbf{S1g} is used in the stronger
	form, and that the safeguarded acceptance rule in \eqref{eq:Omega-hat-safeguard} selects
	\(x^{k+1}\in\widehat{\Omega}^k\) satisfying \eqref{eq:accept-safeguard}. 

	Assume also that the starting point of the main loop satisfies
	\eqref{eq:full-support-alpha-init}, i.e.,
	\begin{equation*}
		\|x^0\|_0=s,
		\qquad
		\min_{i\in I_1(x^0)} |x_i^0|\ge \alpha.   
	\end{equation*}
	Then there exist \(K\ge0\) and \(T^\ast\subseteq[n]\), with \(|T^\ast|=s\),
	such that
	\[
	I_1(x^k)=T^\ast,
	\qquad
	\min_{i\in T^\ast}|x_i^k|\ge\alpha,
	\qquad
	\forall k\ge K.
	\]
	Consequently, every accumulation point \(\bar x\) of the generated sequence
	satisfies
	\[
	I_1(\bar x)=T^\ast,
	\qquad
	\|\bar x\|_0=s.
	\]
\end{lem}

\begin{proof}
	We first show that the safeguarded admissible set in
	\eqref{eq:Omega-hat-safeguard} is nonempty at every iteration.
	By the initialization assumption \eqref{eq:full-support-alpha-init}, $\|x^0\|_0=s$ and $\min_{i\in I_1(x^0)}|x_i^0|\ge\alpha$. Since \(x^0\in\Omega^0\), the candidate \(y=x^0\) satisfies the full-support
	and active-component lower-bound requirements in
	\eqref{eq:Omega-hat-safeguard}. Moreover, it satisfies the same-support
	alternative, because $I_1(y)=I_1(x^0)$. Therefore, \(x^0\in\widehat{\Omega}^0\).
	
	Assume inductively that the condition \eqref{eq:accept-safeguard} holds at $x^k$, i.e., $\|x^k\|_0=s$ and $\min_{i\in I_1(x^k)}|x_i^k|\ge\alpha$. Since \(x^k\in\Omega^k\), the candidate \(y=x^k\) satisfies the full-support and active-component lower-bound requirements in \eqref{eq:Omega-hat-safeguard}. It also satisfies the same-support alternative, because $I_1(y)=I_1(x^k)$. Hence, \(x^k\in\widehat{\Omega}^k\), so
	\(\widehat{\Omega}^k\neq\emptyset\), and the safeguarded acceptance rule is
	well defined. Since \(x^{k+1}\in\widehat{\Omega}^k\), the definition of
	\(\widehat{\Omega}^k\) gives the satisfaction of \eqref{eq:accept-safeguard} at $x^{k+1}$, i.e.,
	$\|x^{k+1}\|_0=s$ and $
	\min_{i\in I_1(x^{k+1})}|x_i^{k+1}|\ge\alpha$. By induction, every iterate has full support and all active components are bounded
	below by \(\alpha\).
	
	Suppose, by contradiction, that infinitely many support changes occur. Let
	\(\{k_\ell\}_{\ell\ge0}\) be the strictly increasing sequence of indices such
	that $I_1(x^{k_\ell+1})\neq I_1(x^{k_\ell})$. By the safeguard in \eqref{eq:Omega-hat-safeguard}, every accepted support-changing step satisfies
	\[
	f(x^{k_\ell})-f(x^{k_\ell+1})
	\ge
	\delta_{\rm supp}.
	\]
	Since the accepted objective values are nonincreasing, for every \(N\ge0\),
	\[
	f(x^0)-f(x^{k_N+1})
	\geq
	\sum_{\ell=0}^{N}
	\bigl(f(x^{k_\ell})-f(x^{k_\ell+1})\bigr)
	\ge
	(N+1)\delta_{\rm supp}.
	\]
	The left-hand side is bounded above because \(f\) is bounded below on
	\(\mathcal L(x^0)\), whereas the right-hand side tends to \(+\infty\). This is impossible. Hence only finitely many support changes occur. Therefore, there exist \(K\ge0\) and \(T^\ast\subseteq[n]\) such that $I_1(x^k)=T^\ast$, for all $k\ge K$. Since every iterate has full support and active components bounded below by
	\(\alpha\), we obtain $|T^\ast|=s$, $\min_{i\in T^\ast}|x_i^k|\ge\alpha$ for all $k\ge K$.
	Let \(\bar x\) be any accumulation point. Then there exists a subsequence
	\(x^{k_\ell}\to\bar x\). For all sufficiently large \(\ell\), $I_1(x^{k_\ell})=T^\ast$ and $\min_{i\in T^\ast}|x_i^{k_\ell}|\ge\alpha$. Passing to the limit yields $|\bar x_i|\ge\alpha$, for all $i\in T^\ast$. Thus, \(T^\ast\subseteq I_1(\bar x)\). Since \(\bar x\in C\cap C_s\) and
	\(|T^\ast|=s\), we obtain $I_1(\bar x)=T^\ast$ and $\|\bar x\|_0=s$. 
\end{proof}
\begin{cor}[ZCW stationarity of all accumulation points under the safeguard]
	\label{cor:all-accumulation-zcw-safeguard}
	Assume that the hypotheses of Theorem~\ref{thm:RZCW-PSS-ZCW} hold. Suppose also
	that the optional support-identification safeguard in Step~\textbf{S1g} is
	imposed and that the starting point of the main loop satisfies \eqref{eq:full-support-alpha-init}. Then every accumulation point of the sequence generated by {\tt RZCW-PSS}
	is a  ZCW stationary point almost surely.
\end{cor}

\begin{proof}
	By Lemma~\ref{lem:eventual-full-support-stabilization}, every accumulation point
	has full support. Therefore every accumulation point is covered by
	Theorem~\ref{thm:RZCW-PSS-ZCW}. Hence every accumulation point is 
	ZCW stationary almost surely.

\end{proof}

\begin{rem}[Role of the support-identification safeguard]
	Theorem~\ref{thm:RZCW-PSS-ZCW} establishes  ZCW stationarity for
	all full-support accumulation points. The optional safeguard in
	Step~\textbf{S1g}, together with full-support initialization, ensures that all
	accumulation points are eventually reached along full-support iterates. Therefore,
	Corollary~\ref{cor:all-accumulation-zcw-safeguard} upgrades the conclusion from
	full-support accumulation points to all accumulation points of the generated
	sequence.
\end{rem}

\subsection{Convergence Rate}\label{sec:ConRate}

While global linear convergence for cardinality-constrained problems is
generally precluded by the combinatorial nature of support selection,
we establish that {\tt RZCW-PSS} enjoys a \emph{conditional local linear
	convergence rate} once an active support has stabilized.

For general cardinality-constrained optimization problems, global convergence
rates are not available in the literature due to the nonconvex and combinatorial
nature of the feasible set. Existing analyses typically establish convergence to
stationary points without quantifying the rate. Convergence rates arise only
locally, once the active support has been identified or under additional
regularity assumptions.

Lemma~\ref{lem:eventual-full-support-stabilization} shows that, under the
support-identification safeguard in Step~\textbf{S1g} and full-support
initialization, the generated sequence eventually remains on one fixed
full support. Hence, in the safeguarded implementation, the support-stabilization
hypothesis in the local-rate theorem is no longer an independent assumption.
The only remaining local trajectory assumption is eventual containment in the
neighborhood \(\mathcal N\) where the restricted regularity assumptions hold.

\subsubsection{Theoretical Assumptions}

We analyze the local behavior of the algorithm in a neighborhood of a local
minimizer \(x^\ast\in\Omega\), where \(\Omega=C\cap C_s\), under regularity
conditions on a fixed active support
\begin{equation}\label{eq:fixsup}
	T^\ast:=I_1(x^\ast).   
\end{equation}
Define the restricted feasible set
\begin{equation}\label{eq:resfeas}
	C_{T^\ast}
	:=
	\{w\in\mathbb R^{|T^\ast|}:U_{T^\ast}w\in C\}.  
\end{equation}
For \(x\) satisfying \(I_1(x)\subseteq T^\ast\), let \(x_{T^\ast}\) denote the restriction of \(x\) to the index set \(T^\ast\). Then \(x_{T^\ast}\in C_{T^\ast}\). Equivalently, we use the lifted fixed-support feasible set
\begin{equation}\label{eq:lifted-fixed-support-set}
	\mathcal M_{T^\ast}
	:=
	\{x\in C:I_1(x)\subseteq T^\ast\}
	=
	\{U_{T^\ast}w:w\in C_{T^\ast}\}.
\end{equation}
Thus, \(C_{T^\ast}\) is the restricted feasible set in the reduced coordinates
\(w\in\mathbb R^{|T^\ast|}\), while \(\mathcal M_{T^\ast}\) is the
corresponding feasible set in the original variable space \(\mathbb R^n\).
Consequently,
\[
x\in\mathcal M_{T^\ast}
\quad\Longleftrightarrow\quad
x=U_{T^\ast}x_{T^\ast}
\quad \text{and}\quad
x_{T^\ast}\in C_{T^\ast}.
\]

\begin{ass}[Restricted Strong Convexity]
	\label{ass:RSC}
	The objective function \(f\) is twice continuously differentiable.
	There exist constants \(0<\mu\le L<\infty\) and a neighborhood
	\(\mathcal N\) of \(x^\ast\) such that
	\[
	\mu I \preceq \nabla^2_{T^\ast} f(x) \preceq L I,
	\qquad
	\forall x\in \mathcal N
	\ \text{with}\ 
	I_1(x)\subseteq T^\ast .
	\]
\end{ass}

Assumption~\ref{ass:RSC} is a local restricted regularity condition, not a global
convexity assumption on the original cardinality-constrained problem. It is
imposed only after the support has stabilized at \(T^\ast\), where the algorithm
locally behaves like a method applied to the smooth constrained problem on
\(C_{T^\ast}\). The lower Hessian bound provides local curvature, while the upper
bound gives the Lipschitz constant needed for the projected-gradient descent
estimate.

\begin{ass}[Restricted projected-gradient error bound]
	\label{ass:pg-nondeg}
	With the same constant \(L\) as in Assumption~\ref{ass:RSC}, there exist
	\(\mu_{\rm pg}>0\) and a neighborhood \(\mathcal N\) of \(x^\ast\) such that,
	for all \(x\in\mathcal N\) with \(I_1(x)\subseteq T^\ast\), the restricted
	projected-gradient mapping
	\[
	G_L^{T^\ast}(x)
	:=
	L\left[
	x_{T^\ast}
	-
	\proj_{C_{T^\ast}}
	\left(
	x_{T^\ast}
	-
	\frac1L g_{T^\ast}(x)
	\right)
	\right]
	\]
	satisfies $\|G_L^{T^\ast}(x)\|^2
	\ge
	2\mu_{\rm pg}\bigl(f(x)-f(x^\ast)\bigr)$.
\end{ass}

Assumption~\ref{ass:pg-nondeg} is the corresponding local error-bound condition
for the restricted projected-gradient mapping. It connects the computable
projected-gradient decrease on the fixed support to the objective gap
\(f(x)-f(x^\ast)\). In this sense, it is a restricted fixed-support analogue of
the projected-gradient error-bound conditions used in linear convergence
analyses of projected and proximal-gradient methods; see, e.g.,
\cite[Definition~5 and Theorems~6--7, Pages 12--13]{NecoaraNesterovGlineur2019}.
Without such an error-bound condition, or an equivalent quadratic-growth property, one may still obtain descent, but the descent need not be
proportional to the current objective gap, and a linear rate cannot be
concluded.

\begin{ass}[Local candidate richness]
	\label{ass:candidate-richness}
	Let \(T^\ast\) be the stabilized support and
	\begin{equation}\label{eq:resprogra}
		\tilde x^k
		:=
		U_{T^\ast}
		\proj_{C_{T^\ast}}
		\left(
		x^k_{T^\ast}
		-
		\frac1L g_{T^\ast}(x^k)
		\right)
	\end{equation}
	be the restricted projected-gradient trial point. We assume that, for all
	sufficiently large \(k\), the candidate set \(\Omega^k\) contains a point
	\(\bar x^k\in\Omega^k\) satisfying $f(\bar x^k)\le f(\tilde x^k)$.
\end{ass}

Assumption~\ref{ass:candidate-richness} is a local candidate-set richness
condition. Its role is analogous to the sufficient decrease requirements used in
trust-region methods, where an approximate subproblem solution is required to
produce at least a fixed fraction of the decrease obtained by a reference step,
such as the Cauchy step~\cite{ConnGouldToint2000}. Indeed, the condition
\(f(\bar x^k)\le f(\tilde x^k)\) is equivalent to $f(x^k)-f(\bar x^k)
\ge
f(x^k)-f(\tilde x^k)$, so the candidate \(\bar x^k\) provides at least the decrease of the
reference restricted projected-gradient trial point \(\tilde x^k\).

\subsubsection{Local Rate of Convergence}

In this subsection, we analyze the local convergence behavior of the proposed
{\tt RZCW-PSS} algorithm after a fixed support has been identified and
maintained. Once this occurs, the problem reduces locally to a smooth constrained
optimization problem over the restricted feasible set \(C_{T^\ast}\). We first
establish a quantitative descent estimate under inexact subproblem solutions,
showing that sufficient decrease is preserved despite numerical errors. Building
on this result and the restricted projected-gradient error bound, we then prove
a linear rate for the objective values.

The local convergence analysis relies only on
Assumption~\ref{ass:subproblem}(ii), which is standard in inexact descent methods.
Assumption~\ref{ass:subproblem}(i) is included to cover broader inexactness models
in the global convergence analysis.

\begin{prop}[Descent under inexact subproblem solutions]
	\label{prop:inexact-descent}
	Suppose that Assumptions~\ref{ass:subproblem}--~\ref{ass:candidate-richness} hold.
	Assume that the support-identification safeguard in Step~\textbf{S1g} is imposed
	with full-support initialization, and that the stabilized support given by
	Lemma~\ref{lem:eventual-full-support-stabilization} is \(T^\ast\). Suppose
	further that, for all sufficiently large \(k\), the iterates \(x^k\) and the
	restricted projected-gradient trial points \(\tilde x^k\) defined in \eqref{eq:resprogra} belong to
	the neighborhood \(\mathcal N\). Assume also that, for all sufficiently large \(k\), the candidate set
	\(\Omega^k\) contains a point \(\bar x^k\in\Omega^k\) satisfying $f(\bar x^k)\le f(\tilde x^k)$.
	Finally, suppose that the iterates satisfy \eqref{eq:accept-no-safeguard} and that $\xi^k\le \eta\delta^k$
	for some $\eta\in(0,1)$ (Assumption~\ref{ass:subproblem} (ii)). Then there exists \(k_0\in\mathbb N\) such that, for all \(k\ge k_0\), the actual decrease in \eqref{e.defDeltak} satisfies
	\[
	\Delta^k
	\ge
	\frac{1-\eta}{2L}
	\left\|G_L^{T^\ast}(x^k)\right\|^2 .
	\]
\end{prop}

\begin{proof}
	
	By Lemma~\ref{lem:eventual-full-support-stabilization}, eventual containment in
	\(\mathcal N\) and there exists \(k_0\) such
	that all required conditions hold for every \(k\ge k_0\). In particular, after
	increasing \(k_0\) if necessary, the support is fixed at \(T^\ast\), the iterates
	belong to \(\mathcal N\), and the candidate-richness condition (Assumption~\ref{ass:candidate-richness}) holds. Hence, for all \(k\ge k_0\), the iterates remain in
	\(\mathcal M_{T^\ast}\) defined in \eqref{eq:lifted-fixed-support-set}. Let $w^k:=x^k_{T^\ast}$ and $g^k:=g_{T^\ast}(x^k)$. Consider the restricted projected-gradient step
	\[
	\tilde w^k
	:=
	\proj_{C_{T^\ast}}
	\left(
	w^k-\frac1L g^k
	\right),
	\qquad
	\tilde x^k:=U_{T^\ast}\tilde w^k.
	\]
	Then \(\tilde x^k\in\mathcal M_{T^\ast}\), where $\mathcal M_{T^\ast}$ is from \eqref{eq:lifted-fixed-support-set}. By the candidate-richness assumption,
	there exists \(\bar x^k\in\Omega^k\) such that $f(\bar x^k)\le f(\tilde x^k)$. Since the support is fixed at \(T^\ast\), the local analysis is carried out on
	the restricted feasible set \(C_{T^\ast}\), which comes from \eqref{eq:resfeas}. Define the restricted function
	\[
	\varphi(w):=f(U_{T^\ast}w),
	\qquad w\in C_{T^\ast}.
	\]
	Since $\nabla^2\varphi(w)
	=
	U_{T^\ast}^{\top}\nabla^2 f(U_{T^\ast}w)U_{T^\ast}
	=
	\nabla^2_{T^\ast}f(U_{T^\ast}w)$,  Assumption~\ref{ass:RSC} yields
	\[
	\mu I \preceq \nabla^2\varphi(w) \preceq LI,
	\]
	in the neighborhood under consideration. Hence \(\varphi\) is
	\(\mu\)-strongly convex and its gradient is \(L\)-Lipschitz. So, the standard descent lemma (cf. \cite[Lemma~4.22]{beck2014introduction}) applied
	to \(\varphi\) gives
	\[
	f(\tilde x^k)
	=
	\varphi(\tilde w^k)
	\le
	\varphi(w^k)
	+
	\langle \nabla \varphi(w^k),\tilde w^k-w^k\rangle
	+
	\frac{L}{2}\|\tilde w^k-w^k\|^2.
	\]
	Since $\nabla\varphi(w^k)=g_{T^\ast}(x^k)=g^k$, this becomes
	\[
	f(\tilde x^k)
	\le
	f(x^k)
	+
	\langle g^k,\tilde w^k-w^k\rangle
	+
	\frac{L}{2}\|\tilde w^k-w^k\|^2.
	\]
	
	Moreover, \(\tilde w^k\) is the projection of
	\(w^k-\frac1L g^k\) onto the closed convex set \(C_{T^\ast}\). So, the
	variational characterization of the projection, \cite[Theorem~3.16]{Bauschke2017}, yields
	\[
	\left\langle
	w^k-\frac1L g^k-\tilde w^k,\,
	w-\tilde w^k
	\right\rangle
	\le 0,
	\qquad
	\forall w\in C_{T^\ast}.
	\]
	Taking \(w=w^k\) yields $\langle g^k,\tilde w^k-w^k\rangle
	\le
	-L\|\tilde w^k-w^k\|^2$. Therefore,
	\[
	f(\tilde x^k)
	\le
	f(x^k)
	-
	\frac{L}{2}\|\tilde w^k-w^k\|^2.
	\]
	Since $G_L^{T^\ast}(x^k)
	=
	L(w^k-\tilde w^k)$, we obtain $f(x^k)-f(\tilde x^k)
	\ge
	\frac{1}{2L}
	\left\|G_L^{T^\ast}(x^k)\right\|^2$.By Assumption~\ref{ass:candidate-richness}, there exists
	\(\bar x^k\in\Omega^k\) such that \(f(\bar x^k)\le f(\tilde x^k)\). Hence,
	\[
	\min_{z\in\Omega^k} f(z)
	\le
	f(\bar x^k)
	\le
	f(\tilde x^k).
	\]
	Therefore, considering \eqref{e.defdeltak},
	\[
	\delta^k
	=
	f(x^k)-\min_{z\in\Omega^k}f(z)
	\ge
	f(x^k)-f(\tilde x^k)
	\ge
	\frac{1}{2L}
	\left\|G_L^{T^\ast}(x^k)\right\|^2.
	\]
	By \eqref{eq:actdec}, $\Delta^k\ge \delta^k-\xi^k$. Using \(\xi^k\le\eta\delta^k\) (Assumption~\ref{ass:subproblem} (ii)), we obtain
	\[
	\Delta^k
	\ge
	(1-\eta)\delta^k
	\ge
	\frac{1-\eta}{2L}
	\left\|G_L^{T^\ast}(x^k)\right\|^2.
	\]
	This completes the proof.
\end{proof}

Assumptions~\ref{ass:RSC} and~\ref{ass:pg-nondeg} are local regularity
assumptions on the restricted problem associated with the limiting support. They
do not, by themselves, imply convergence of the iterates into the neighborhood
\(\mathcal N\). However, when the support-identification safeguard in
Step~\textbf{S1g} is imposed together with full-support initialization,
Lemma~\ref{lem:eventual-full-support-stabilization} guarantees finite
stabilization on a full support. Hence, in the safeguarded implementation, the
only remaining trajectory condition for the local rate result is eventual
containment in \(\mathcal N\).

\begin{thm}[Conditional local linear convergence under the safeguard]
	\label{thm:local-linear}
	Let \(x^\ast\in\Omega\) be a local minimizer of \(f\) on the restricted feasible
	set \eqref{eq:resfeas} associated with \(T^\ast\) in \eqref{eq:fixsup}.
	Suppose that Assumptions~\ref{ass:RSC} and~\ref{ass:pg-nondeg} hold in a
	neighborhood \(\mathcal N\) of \(x^\ast\), chosen sufficiently small so that
	\[
	f(x)\ge f(x^\ast),
	\qquad
	\forall x\in \mathcal N\cap\mathcal M_{T^\ast}.
	\]
	Assume that the support-identification safeguard in Step~\textbf{S1g} is imposed
	and that the starting point of the main loop satisfies $\|x^0\|_0=s$ and $\min_{i\in I_1(x^0)}|x_i^0|\ge \alpha$ (see, \eqref{eq:full-support-alpha-init}). Suppose further that the stabilized support given by
	Lemma~\ref{lem:eventual-full-support-stabilization} is \(T^\ast\), and that the
	iterates eventually enter the neighborhood \(\mathcal N\); that is, there exists
	\(k_0\) such that $x^k\in\mathcal N$, for all $k\ge k_0$. Assume also that the iterates satisfy the inexactness condition
	\eqref{eq:accept-no-safeguard} with an error sequence \(\{\xi^k\}_{k\geq 0}\) satisfying
	Assumption~\ref{ass:subproblem}; in particular,
	\[
	\xi^k\le \eta\delta^k
	\quad
	\text{for some } \eta\in(0,1).
	\]
	Moreover, assume that the local candidate-richness condition in
	Assumption~\ref{ass:candidate-richness} holds. Then there exists a constant \(\rho_{\rm lin}\in(0,1)\) such that
	\[
	f(x^{k+1})-f(x^\ast)
	\le
	(1-\rho_{\rm lin})\bigl(f(x^k)-f(x^\ast)\bigr),
	\qquad
	\forall k\ge k_0,
	\]
	where one may take
	\begin{equation}\label{eq:rho-lin}
		\rho_{\rm lin}
		:=
		\min\left\{
		\frac12,\,
		\frac{(1-\eta)\mu_{\rm pg}}{L}
		\right\}.
	\end{equation}
\end{thm}

\begin{proof}
	By Lemma~\ref{lem:eventual-full-support-stabilization}, the support is fixed at
	\(T^\ast\) for all sufficiently large \(k\). Since the iterates eventually enter
	\(\mathcal N\), we may increase \(k_0\), if necessary, so that for all
	\(k\ge k_0\),
	\[ 
	x^k\in \mathcal N\cap\mathcal M_{T^\ast}.
	\]
	By the choice of \(\mathcal N\), we also have $f(x^k)-f(x^\ast)\ge0$ for all $k\ge k_0$. By Proposition~\ref{prop:inexact-descent}, for all \(k\ge k_0\),
	\[
	\Delta^k
	\ge
	\frac{1-\eta}{2L}
	\left\|G_L^{T^\ast}(x^k)\right\|^2 .
	\]
	By the definition of the actual accepted decrease in \eqref{e.defDeltak}, $\Delta^k=f(x^k)-f(x^{k+1})$. Therefore,
	\[
	f(x^{k+1})
	=
	f(x^k)-\Delta^k
	\le
	f(x^k)
	-
	\frac{1-\eta}{2L}
	\left\|G_L^{T^\ast}(x^k)\right\|^2.
	\]
	By Assumption~\ref{ass:pg-nondeg},
	\[
	\left\|G_L^{T^\ast}(x^k)\right\|^2
	\ge
	2\mu_{\rm pg}\bigl(f(x^k)-f(x^\ast)\bigr).
	\]
	Combining the last two inequalities gives
	\[
	f(x^{k+1})
	\le
	f(x^k)
	-
	\frac{(1-\eta)\mu_{\rm pg}}{L}
	\bigl(f(x^k)-f(x^\ast)\bigr).
	\]
	Using \eqref{eq:rho-lin} and since \(\eta\in(0,1)\), \(\mu_{\rm pg}>0\), and \(L>0\), we obtain
	\[
	0<\rho_{\rm lin}<1
	\qquad\text{and}\qquad
	\rho_{\rm lin}\le \frac{(1-\eta)\mu_{\rm pg}}{L}.
	\]
	Since \(f(x^k)-f(x^\ast)\ge0\), it follows that $f(x^{k+1})
	\le
	f(x^k)
	-
	\rho_{\rm lin}\bigl(f(x^k)-f(x^\ast)\bigr)$. Subtracting \(f(x^\ast)\) from both sides yields
	\[
	f(x^{k+1})-f(x^\ast)
	\le
	(1-\rho_{\rm lin})\bigl(f(x^k)-f(x^\ast)\bigr).
	\]
	This proves the claimed local linear rate.

\end{proof}

\begin{rem}[Two-phase convergence behavior]
	Theorem~\ref{thm:local-linear} describes the local refinement phase of
	{\tt RZCW-PSS} conditional on stabilization of the active support.
	The ZCW-aware reservoir injection mechanism (Step~{\bf S1e}) is used in the
	global phase to generate ZCW super supports and the associated restricted
	subproblem candidates. Once the iterates remain on a fixed support,
	the algorithm transitions to a deterministic restricted refinement phase in which
	projected subspace optimization yields linear convergence under the assumptions
	above.
\end{rem}

We note that the practical implementation includes additional components
(e.g., {\tt BFS} ranking and restricted {\tt FISTA} solves) that preserve
feasibility and sufficient decrease. These modifications do not affect the local
convergence analysis once the support has stabilized.

\section{Numerical Experiments}\label{sec:num}

This section evaluates the practical performance of the proposed {\tt RZCW-PSS}
algorithm on a diverse collection of randomly generated and data-driven,
cardinality-constrained optimization problems.
We compare four variants of {\tt RZCW-PSS} with established methods, namely
{\tt PSS}, {\tt BFS}, and {\tt ZCWS}, with respect to efficiency, robustness,
objective quality, and structural support similarity.
The experiments are designed to assess whether the methods can obtain
high-quality feasible sparse solutions within a prescribed computational budget,
rather than to report an exact numerical certificate of  ZCW
stationarity.
A total of 50 test problems spanning a wide range of dimensions and
sparsity regimes are considered; detailed descriptions of the benchmark
generation procedure and problem classes are provided in
Section~9 of \cite{suppMat}.

\subsection{Efficiency and Robustness}

The two efficiency measures considered  are {\tt nf2g}={\tt nf}+2{\tt ng} (where \texttt{nf} is the total number of function evaluations and \texttt{ng} is the gradient evaluations) and the computational time in seconds (\texttt{sec}). A common motivation for adopting {\tt nf2g} as a cost measure is its ability to reflect the disproportionate cost of gradient evaluations relative to function calls. Since, in many practical settings, computing gradients dominates the overall expense, emphasizing their contribution leads to a realistic assessment of computational effort. The performance profile (\cite[Section 10]{suppMat})  indicates which solver is the {\bf most efficient}, in terms of lowest cost according to these measures, as well as the {\bf most robust}, measured by the largest number of problems successfully solved.

\subsection{Stopping Criterion, Relaxed Support Recovery, and Solution Quality}
\label{sec:stop-quality}

The numerical experiments are designed to assess computational efficiency,
robustness, objective quality, and the structural similarity of the sparse
supports recovered by the different methods. Although the convergence analysis in
Section~\ref{sec:conv} is stated in terms of  zero-coordinatewise
stationarity, exact numerical certification of this condition is generally
expensive. Indeed,  ZCW stationarity is defined through
restricted sparse subproblems over ZCW super supports. Therefore, we do not
use a cheap gradient-score test as a numerical stationarity residual.

\bfi{Why exact ZCW certification is not used as a numerical stopping test, and what we do instead.}
The numerical experiments do not attempt to certify Beck-Hallak ZCW
stationarity at every evaluated point. Exact certification would require
checking whether the solution of any restricted sparse subproblem over the relevant ZCW
super supports can improve the current point. In other words, for each generated
or admissible ZCW super support \(T\), one would need to solve, or accurately
approximate \eqref{eq:ressubpro}, namely,
\[
\min\{f(y):y\in C,\ I_1(y)\subseteq T\},
\]
and compare this value with \(f(x)\). Performing this test at every evaluated
point would be computationally expensive and would obscure the practical
comparison of the solvers. 

The algorithm may still use inexpensive first-order ranking information
internally to guide coordinate selection and support-exchange candidates. This
ranking is useful for generating trial points, but it is not used as a numerical
certificate of ZCW stationarity. The absence of an improving candidate generated
by this ranking rule should not be interpreted as proof that all ZCW restricted
subproblems are non-improving. For this reason, the reported numerical evaluation is based on computational
cost, final objective value, and relaxed support recovery. These quantities
directly measure the practical efficiency of the method, the quality of the
solutions obtained, and the structural similarity of the recovered sparse
supports.

\bfi{Objective-quality stopping criterion.}
For the performance profiles, a solver is declared successful on a problem if it
attains a prescribed objective-quality tolerance before reaching the maximum
allowed computational budget. The computational budgets are measured by
\({\tt nf2g}\) and \({\tt sec}\).
The objective-quality measure is the convergence ratio
\[
q_{\mathrm{sol}}
:=
\frac{f_s - f_{\mathrm{opt}}}{f_0 - f_{\mathrm{opt}}},
\]
where \(f_s\) denotes the best function value obtained by solver \(s\), \(f_0\)
is the objective value at the initial point, and \(f_{\mathrm{opt}}\) is the
best-known objective value for the corresponding problem instance. In the
numerical experiments, a run is counted as successful if
\[
q_{\mathrm{sol}}\leq \varepsilon,
\qquad
\varepsilon:=10^{-4},
\]
before the maximum allowed \({\tt nf2g}\) or \({\tt sec}\) is reached. Lower
values of \(q_{\mathrm{sol}}\) indicate better objective quality.

\bfi{Relaxed support recovery.}
Since cardinality-constrained problems are nonconvex, exact support recovery is
not necessarily aligned with objective quality. Different supports may correspond
to distinct local minima, and a solver may obtain a lower objective value using a
support that is not identical to the support found by another method. Therefore,
we use a relaxed support recovery measure to quantify structural similarity
without requiring exact equality of supports.

For a problem instance \(p\) and solver \(s\), let \(x^{(p,s)}\) denote the
solution returned by solver \(s\), and define its support by
\[
\mathcal S^{(p,s)}
:=
\operatorname{supp}\bigl(x^{(p,s)}\bigr).
\]
Given a reference solution \(x^{(p,\mathrm{ref})}\), with support
\(\mathcal S^{(p,\mathrm{ref})}\), the relaxed support recovery rate of solver
\(s\) at threshold \(\theta\in(0,1)\) is defined by
\[
\mathrm{RSR}_{\theta}(s)
:=
\frac{1}{P}
\sum_{p=1}^{P}
\mathbf{1}
\left\{
\frac{
	\left|
	\mathcal S^{(p,s)}
	\cap
	\mathcal S^{(p,\mathrm{ref})}
	\right|
}{
	\left|
	\mathcal S^{(p,\mathrm{ref})}
	\right|
}
\geq \theta
\right\},
\]
where \(P\) is the number of test problems and \(\mathbf{1}\{\cdot\}\) denotes
the indicator function. Unless stated otherwise, we use \(\theta=0.9\).

The relaxed support recovery rate is not an optimality certificate. Rather, it
measures how similar the support identified by a solver is to a reference
support. This is useful for distinguishing two effects: whether a method recovers
structurally similar sparse patterns, and whether it also improves the objective
value. In our experiments, the most informative comparison is therefore the joint
behavior of relaxed support recovery and final objective value. A method that
obtains a lower objective value with a different but substantially overlapping
support is regarded as having found a better sparse solution, even if it does not
exactly reproduce the reference support.

Consequently, the figures in this section report final support size, matched
objective values, and relaxed support recovery rates. We do not report a
stationarity residual in the numerical evaluation, because exact 
ZCW certification would require solving restricted sparse subproblems and is
not computationally practical for the reported experiments.

\subsection{Numerical Comparison}

We begin by comparing four variants of the proposed {\tt RZCW-PSS} algorithm,
which differ only in the orthonormalization procedure used to construct the
randomized subspace bases (see \cite[Section 1]{suppMat}).
Specifically, we consider {\tt RZCW-PSS-orth}, {\tt RZCW-PSS-mgs},
{\tt RZCW-PSS-qr}, and {\tt RZCW-PSS-svd}, corresponding to the use of the built-in
{\tt orth} operator, the modified Gram-Schmidt method, QR factorization, and singular value
decomposition, respectively.

From Figure~11 of \cite{suppMat}, {\tt RZCW-PSS-qr} consistently achieves the best
performance among all variants, exhibiting both the lowest computational cost
(in terms of {\tt nf2g} and CPU time {\tt sec}) and the highest number of solved
instances under the objective-quality criterion
\(q_{\mathrm{sol}}\le \varepsilon=10^{-4}\). Based on this observation, we select
{\tt RZCW-PSS-qr} as the representative implementation for all subsequent
comparisons.

We then compare {\tt RZCW-PSS-qr} with {\tt ZCWS} to evaluate efficiency,
robustness, and objective quality. As shown in Figure~\ref{f.f1},
{\tt RZCW-PSS-qr} is both significantly more efficient and more robust, requiring fewer
evaluations and less computational time while solving a larger fraction of
instances according to the objective-quality criterion. This improvement can be
attributed to the combination of projected {\tt p-usequence} initialization and
randomized subspace exploration, which enables the method to escape poor sparse
local structures and identify lower-objective feasible points.

\begin{figure}[http!]
	\centering
	\begin{tabular}{cc}
		\includegraphics[width=0.48\textwidth,height=0.42\textwidth]{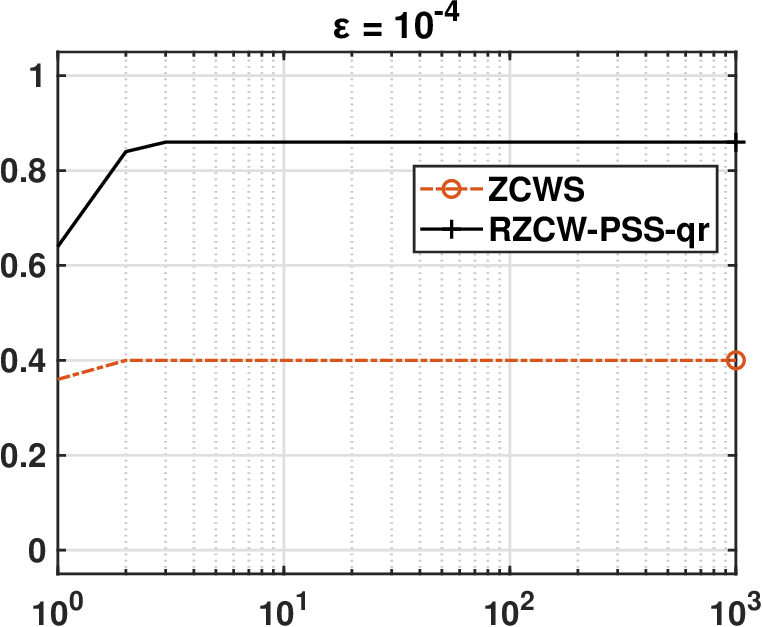}
		&
		\includegraphics[width=0.48\textwidth,height=0.42\textwidth]{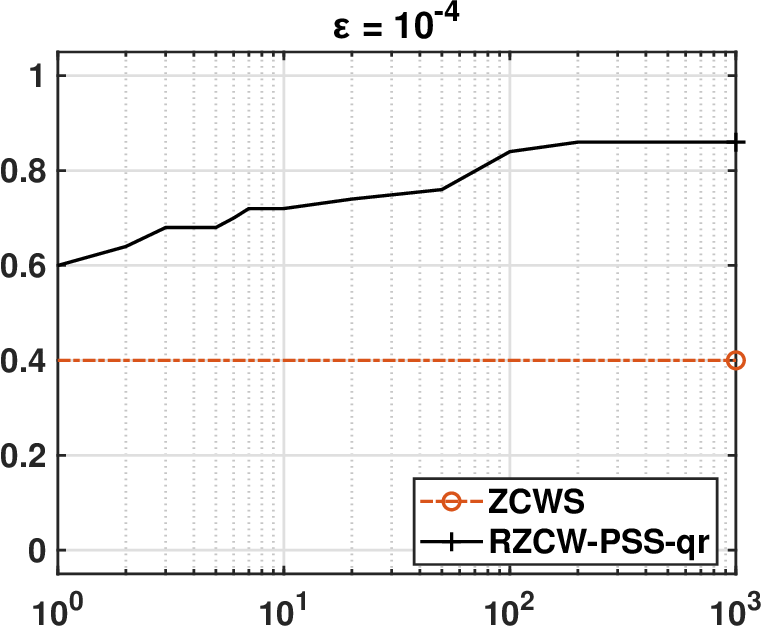}
	\end{tabular}
	\caption{Performance profiles $\{\rho_{s_i}(\tau)\}_{i=1}^{2}$ of the solvers
		$\{s_i\}_{i=1}^{2}=\{${\tt RZCW-PSS-qr}, {\tt ZCWS}$\}$ in terms of {\tt nf2g}
		(left) and {\tt sec} (right), using the objective-quality criterion
		\(q_{\mathrm{sol}}\le \varepsilon=10^{-4}\).}
	\label{f.f1}
\end{figure}

A more detailed analysis of solution quality is provided in
Figure~\ref{f.fsupport}, where we compare relaxed support recovery and objective
value reduction for {\tt ZCWS} and {\tt RZCW-PSS-qr}. The purpose of this
comparison is not to certify optimality or stationarity, but to assess whether
the sparse supports identified by different solvers are structurally related and
whether this support similarity is accompanied by improved objective values.
This distinction is important in nonconvex cardinality-constrained optimization,
where exact support recovery is often too restrictive and may not correlate with
solution quality.

The results show that {\tt RZCW-PSS-qr} attains substantially lower objective
values while maintaining meaningful relaxed support overlap with the reference
solutions. Thus, the method does not merely reproduce the same supports as
{\tt ZCWS}; rather, it often identifies different but structurally related sparse
supports lying in more favorable regions of the objective landscape.
Consequently, the joint use of objective value reduction and relaxed support
recovery provides a more informative assessment than exact support matching
alone. These results support the conclusion that, for nonconvex sparse
optimization, high-quality solutions should be evaluated by both their objective
values and their structural support similarity, rather than by exact
combinatorial recovery.

\begin{figure}[http!]
	\centering
	\begin{tabular}{ccc}
		\includegraphics[width=0.32\textwidth,height=0.26\textwidth]{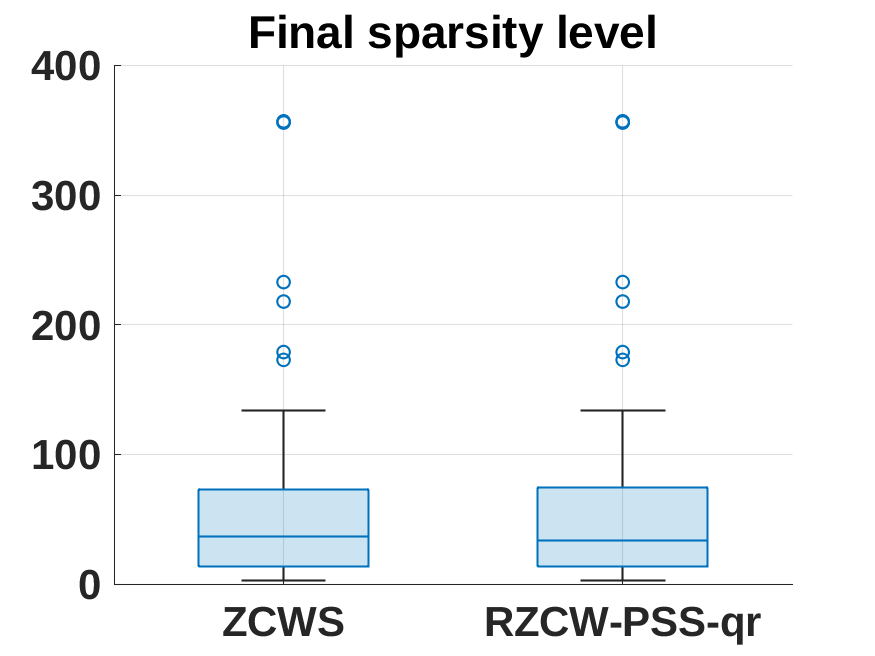}
		&
		\includegraphics[width=0.32\textwidth,height=0.26\textwidth]{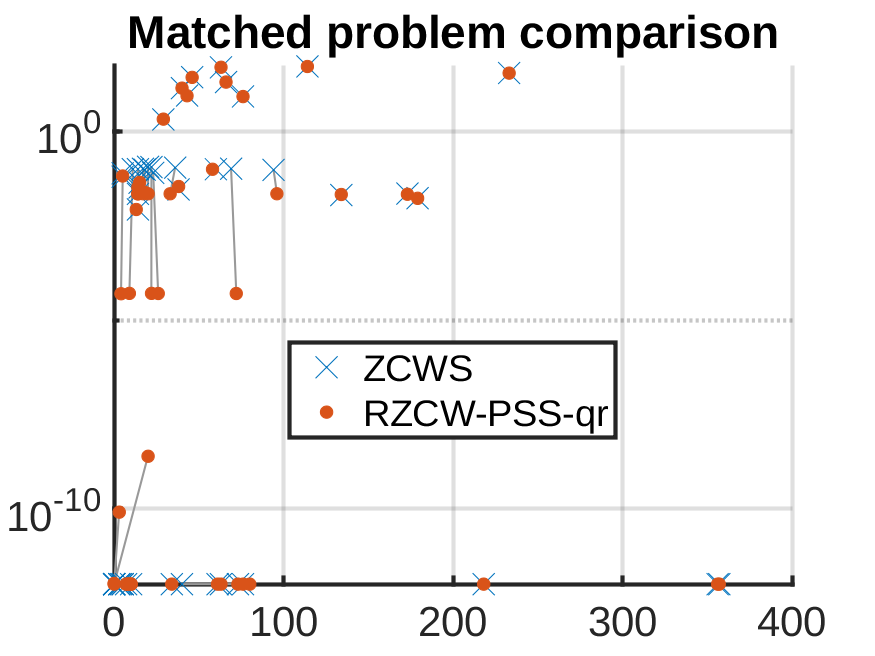}
		&
		\includegraphics[width=0.32\textwidth,height=0.26\textwidth]{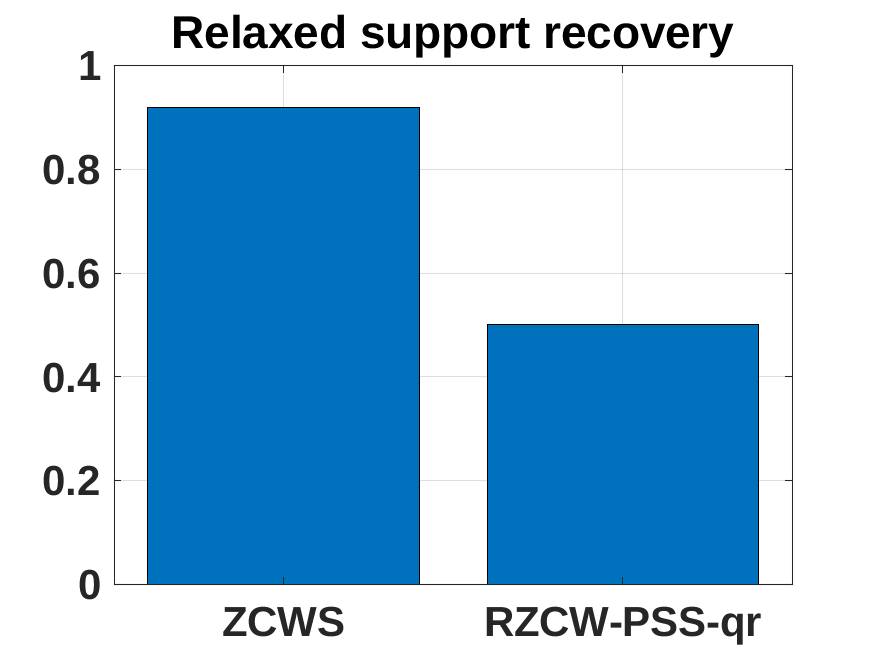}
	\end{tabular}
	\caption{Comparison of {\tt ZCWS} and {\tt RZCW-PSS-qr} under cardinality
		constraints.
		(Left) Final support size achieved by each method.
		(Middle) Matched-problem comparison of the final objective value versus
		support size, where each connected pair corresponds to the same problem
		instance; {\tt RZCW-PSS-qr} consistently attains lower objective values, often
		with different but structurally related supports.
		(Right) Relaxed support recovery rate \(\mathrm{RSR}_{\theta}\), measured with
		respect to the reference support and the threshold \(\theta=0.9\).}
	\label{f.fsupport}
\end{figure}

For completeness, additional comparisons with {\tt PSS} and {\tt BFS} are
reported in \cite[Sections~13~and~14]{suppMat}, respectively, where
{\tt RZCW-PSS-qr} demonstrates similar gains in both efficiency and robustness
under the same objective-quality criterion.

The superior performance of {\tt RZCW-PSS-qr} among the tested variants is mainly
due to the numerical stability of QR factorization, which produces well-conditioned
subspace bases in finite precision arithmetic. While alternative orthonormalization
methods generate equivalent subspaces in exact arithmetic, QR leads to more
reliable behavior in practice.

To further interpret the observed convergence behavior, we note that the plateau
visible in some runs corresponds to phases where coordinatewise and swap-based
moves fail to produce sufficient improvement. During these phases, progress
depends on the ZCW-aware reservoir injection mechanism (Step~\textbf{S1e}),
which explores directions involving inactive coordinates and alternative sparse
supports. Once a useful direction is detected, the algorithm resumes descent,
resulting in the observed improvement after the plateau.

\section{Summary and Conclusion}

We proposed the Reservoir ZCW Projected Subspace Search ({\tt RZCW-PSS})
algorithm for cardinality-constrained optimization, which strengthens
classical coordinatewise and swap-based methods through randomized
low-dimensional sparse subspace exploration. A dynamically maintained reservoir
of feasible points, initialized by symmetry-aware uniform sampling and enriched
through a ZCW-aware injection mechanism, enables the discovery of richer descent
directions while preserving both convex and sparsity feasibility.

We established that, under mild smoothness assumptions, every full-support
accumulation point of the iterates generated by {\tt RZCW-PSS} is a 
zero-coordinatewise stationary point almost surely. This result strengthens
existing guarantees for coordinatewise sparse methods, which typically ensure
only partial coordinatewise or \(L\)-stationarity, and provides a theoretical
basis for the randomized support-exchange mechanism used in the algorithm.

The numerical experiments were designed to assess efficiency, robustness,
objective quality and structural similarity of the recovered sparse supports.
The results show that {\tt RZCW-PSS} consistently attains lower objective values
than competing methods while maintaining meaningful relaxed support overlap with
reference solutions. This confirms that, in nonconvex cardinality-constrained
optimization, high-quality sparse solutions are not necessarily characterized by
exact support recovery. Rather, practically useful solutions may have different
but structurally related supports and significantly better objective values.

The experiments also indicate that randomized sparse subspace exploration and
reservoir-based injection help the method escape poor coordinatewise or
swap-based local structures. Thus, the practical advantage of {\tt RZCW-PSS}
comes from its ability to combine feasible sparse exploration with objective
value improvement, rather than from enforcing exact support recovery or relying
on a simple stationarity residual.

Future work will investigate extensions to stochastic and structured sparsity
settings, adaptive strategies for randomized subspace construction, and more
efficient approximate solvers for restricted sparse subproblems arising in exact
ZCW certification.
\end{sloppypar}

{\bf Supplementary Information}  The online version contains supplementary material available in \cite{suppMat}.

\end{document}